\documentclass[12pt]{amsart}
\usepackage{amsmath,amssymb,amsfonts,latexsym}
\usepackage{mathrsfs}
\usepackage[dvips]{geometry}
\usepackage{array}
\usepackage{epsf}
\usepackage{epsfig}
\usepackage[T1]{fontenc}
\usepackage[ansinew]{inputenc}
\usepackage[english]{babel}
\usepackage{amsrefs}

\newtheorem{theorem}{Theorem}[section]
\newtheorem{proposition}[theorem]{Proposition}
\newtheorem{lemma}[theorem]{Lemma}
\newtheorem{definition}[theorem]{Definition}

\newtheorem{remark}[theorem]{Remark}

\numberwithin{equation}{section}

\newcommand{\R}{{\mathbb R}}
\newcommand{\Z}{{\mathbb Z}}
\newcommand{\C}{{\mathbb C}}

\newcommand{\CA}{{\mathcal{A}}}

\newcommand{\CH}{{\mathcal{H}}}

\newcommand{\CL}{{\mathcal{L}}}
\newcommand{\CN}{{\mathcal{N}}}
\newcommand{\CO}{{\mathcal{O}}}

\newcommand{\CR}{{\mathcal{R}}}

\newcommand{\CV}{{\mathcal{V}}}
\newcommand{\CW}{{\mathcal{W}}}
\newcommand{\CSK}{{\mathcal{S K}}}
\newcommand{\CSH}{{\mathcal{S H}}}

\makeindex

\newcommand{\g}{{\mathfrak g}}
\newcommand{\gl}{{\mathfrak g \mathfrak l}}
\newcommand{\so}{{\mathfrak s \mathfrak o
}}

\begin{document}

\author{M.~Vergne}
\address{Mich\`ele Vergne: Universit\'e Paris 7 Diderot, Institut Math\'ematique de Jussieu, Sophie Germain, case 75205, Paris Cedex 13}
\email{michele.vergne@imj-prg.fr}
\title{Bouquets revisited and equivariant elliptic cohomology}

\date{}

\begin{abstract}

Let $M$  be an even dimensional  spin manifold  with  action of a compact Lie group $G$.
We review  the construction of bouquets of analytic equivariant cohomology classes and of $K$-integration.
 Given an elliptic curve $E$,
we introduce, as in Grojnowski,  elliptic bouquets of germs of holomorphic  equivariant cohomology classes on $M$.
Following Bott-Taubes and Rosu, we
show that integration of an elliptic bouquet is well defined.
 In particular, this imply Witten's rigidity theorem.
  Our systematic use of the Chern-Weil construction of
   equivariant classes allows us to
   reduce proofs to algebraic identities.
   \end{abstract}

\maketitle

\section{Introduction}

Let $G$ be a real compact connected Lie group acting on a compact manifold  $M$. In \cite{duflo-vergne}, we defined
 bouquets $(\alpha_g)_{g\in G}$ of equivariant forms on $M$, satisfying some "descent properties".
 Here we view a bouquet as a section of a sheaf $\CSK_G(M)$ on $G$.
 The bouquet construction  in \cite{duflo-vergne} can be seen as an equivariant De Rham model for
 equivariant $K$-theory.

Assume that $M$ is even dimensional and has a $G$-invariant spin structure.
  We defined  in \cite{duflo-vergne} the direct image of a bouquet,
  in particular integration.
   At any $g\in G$, the integration of $\alpha_g$ is a germ of an analytic function
   at $g$, and the integration of a global section is an analytic function on $G$.
If $\CV\to M$ is a $G$-equivariant complex vector bundle, then $\CV$ gives rise to the bouquet of Chern characters ${\bf bch}(\CV)=
({\rm Ch}(g,\CV))_{g\in G}$, a global section of this sheaf.
The delocalized index formula of Berline-Vergne \cite{ber-ver85}
  shows  that the integration  of this bouquet is
  the equivariant index $\Theta(g)$ of the Dirac operator $D$ twisted by $\CV$
  and this property dictates the notion of integration of bouquets.
  Consider to simplify the case where $G=S^1=\{u\in \C; |u|=1\}$.
  Then $\Theta(u)$ is the restriction to $S^1$ of a holomorphic function on $\C^*$.
   However, the  Atiyah-Bott-Segal fixed point formula for $\Theta(u)$ seems to have poles on $|u|=1$,
   in particular for $u=1$. One of our motivations to introduce bouquets
   and integration was to give a  valid formula for $\Theta(u)$ as a germ  of analytic function at any point $u\in S^1$. This technic of descent, already present in Harish-Chandra,
   is essential in the case of infinite dimensional representations
   to define a formula for characters near any point $g\in G$.

Assume that $M$ is an even dimensional   spin manifold.
Witten \cite{witten} has introduced   a formal series $W_q=\sum_{n\geq 0} q^{n/2}\CV_n$ of complex tensor vector bundles on $M$. The (ordinary) index $\Theta$ of the Dirac operator on $M$ twisted by $W_q$ is
a formal series $\Theta(q)=\sum_{n\geq 0} c_n q^{n/2}$ in $q^{1/2}$.
Assume furthermore that $M$ is provided with an action of
$G=S^1=\{u;|u|=1\}$ preserving the spin structure.
Then the equivariant index $\Theta(q,u)=\sum_{n\geq 0} c_n(u) q^{n/2}$
is now a formal series of functions of $u\in \C^*$.
 Witten  conjectured that in fact $\Theta(q,u)$
 does not depend of $u$.
 This is
 the famous  Witten rigidity conjecture, soon proved by several authors notably  Bott-Taubes \cite{bott-taubes}. Let us outline their approach.
Fix $q=e^{2i\pi \tau}$ with $\tau$ in the upper half plane.
Let $L_\tau=(\Z+\Z \tau)$ be the corresponding  lattice in $\C$.
  We identify the Lie algebra $\g$ of $S^1$ with $\R$,  with exponential map $\g\to S^1$ given by
  $x\mapsto e^{2i\pi x}$ and
  $\g_\C$ with $\C$ with exponential map $\g_\C \to G_\C=\C^*$ given by
  $z\mapsto e^{2i\pi z}$.
  If $z\in \C$, then it is easy to see that the series
  $z\mapsto \Theta(q,e^{2i\pi z})$
   defines a meromorphic function on  the elliptic curve
  $E_{\tau}=\C/L_{\tau}$,
   with possible poles  on a finite set of "special points" of finite order.
   The transfer formula of Bott-Taubes shows that this
   function in fact has no poles. Since $E_{\tau}$ is compact, $z\mapsto \Theta(q,e^{2i\pi z})$ does not depend of $z$.

When $G=S^1$ is acting on  $M$,
 Ioanid Rosu \cite{rosu}  considered Grojnowski
  sheaf \cite {groj} of  equivariant elliptic cohomology ${\mathcal Ell}_G(M,\tau)$.  A section of this sheaf
 will be called an "elliptic bouquet".
When $M$ is a point,  then ${\mathcal Ell}_G(M,\tau)$ is the sheaf $\CO(E_\tau)$.
Inspired by Bott-Taubes proof, Rosu defined the $\tau$-integration of  an elliptic bouquet    using a generalized transfer formula,
which insures consistency of integration at different points.
If the section is  defined over an open set $U\subset E_\tau$, its integral is
an holomorphic function   on   $U$.
The $\tau$-integration of a global section is a holomorphic function defined everywhere
on $E_{\tau}$, thus is a constant $c_\tau$.
By considering the global section ${\bf w}=1$, this gives a proof of the rigidity theorem of Witten for the series $W_q$.
Although Rosu's proof (that we basically reproduce here)
might be considered as less elegant than Liu's approach \cite{Liu},
it is (in principle) a finer
result since it allows local integration of any elliptic
bouquet on a {\bf fixed} elliptic curve, and not only of the global bouquet ${\bf w}=1$.
Rosu defined also direct images of elliptic bouquets for fibrations $\pi:M\to B$ .
In particular, applying it again to ${\bf w}=1$,
it  implies the generalization of Witten rigidity theorem
for families at the level of Chern  character proved by Liu and Ma \cite{LiuMa}.
However the  result of rigidity in $K$-theory  \cite{LiuMaZhang} is
outside the reach of equivariant elliptic cohomology technics.

The purpose of this note is to review, and somewhat simplify,
  Rosu's construction of elliptic bouquets and integration.

 We give models of
sheafs for equivariant cohomology,  equivariant $K$-theory, and equivariant elliptic
cohomology.
We have restricted ourselves to define push-forward $M\to {\rm point}$, that is integration for these different theories. The same technic would work for push forward $M\to B$.
 We  make a systematic use of the Chern Weil map in equivariant cohomology,
in particular integration of  our equivariant classes gives germs of holomorphic functions on $E_\tau$ by construction.

We start by reviewing integration on the sheaf $\CSK_G(M)$ since the proof
of the consistency of integration of elliptic bouquets is very similar. Using the BV localization formula
\cite{BV82},
the only delicate point is to check   the compatibility of various choices of orientations and square roots of invariant  functions (such as  Pfaffians).
Similarly, in the elliptic case, our main theorem is Theorem  \ref{th:MainformulaElliptic}
which is an equality of invariant
functions (depending of an element $\gamma\in E_\tau$, and inspired by   Witten series $W_q$) on  $\so(N)$ and is purely algebraic.
Using again BV localization formula, it implies Bott-Taubes-Rosu transfer
formula. Indeed,  with $N=\dim M$,  it allows to glue the integration of germs of an elliptic bouquet on $M$ at different points
$\gamma$ of $E_\tau$.

\section*{Acknowledgements}

Discussions with Eric Vasserot were useful to understand the similarities
of the different  cohomological theories, and this is the point that I want to emphasize in this text. I am thankful to Xiaonan Ma,
 Daniel Berwick-Evans, and   Michel Duflo for providing references and  correcting earlier versions.
 The possible remaining mistakes are mine.


\section{Germs of equivariant forms and bouquets}
To simplify the exposition, in this article,  we assume  that $G$ is a torus, and $M$ compact oriented.

We recall in this section the definition of bouquets of equivariant forms for analytic coefficients.
Various versions of bouquets in more general situations are given in \cite{duflo-vergne}\,,\cite {vergneecm}.

For characteristic classes, we adopt the conventions of \cite{BGV}.
In particular the annoying's $2\pi$ occur in integration formulae and not in
the definition of characteristic classes associated to  connections.

{\bf Let $G$ be a compact torus  with Lie algebra $\g$}.
Let $M$ be a $G$-manifold. We denote by $S_M=\frac{d}{d\epsilon}
\exp(-\epsilon S)\cdot m|_{\epsilon=0}$ the vector field on $M$
produced by the infinitesimal action of $-S\in\g$.
We denote by $M^G$ the set of fixed points of the action of $G$ in $M$.

 If $s\in G$, we denote by $M(s)$ the fixed point of $s$ in $M$.
 If $S\in \g$, we denote by $M(S)$ the set of zeroes of the vector field $S_M$.
 Let $S\in \g$, $s\in G$, then if
 $Y\in \g$ is sufficiently small,
$M(S+Y)=M(S)\cap M(Y)$, and $M(s \exp Y)=M(s)\cap M(Y)$.
Since $G$ is abelian, $M(s)$ and $M(S)$ are $G$-invariant.

  Let $\CA(M)$ be the space of complex valued differential forms on $M$.
  We consider it as a $\Z/2\Z$ graded space.
  Define the space $\CA_{G}(\g,M)$ of $G$-equivariant forms
 as the space of $G$-equivariant maps $X\mapsto \alpha(X)$ from $\g$ to $\CA(M)$
which depends analytically  of $X\in \g$. In our present abelian case, these are maps $\alpha: \g\to \CA(M)^G$.
 Let $d$ be the de Rham differential, and $\iota(X_M)$ be the contraction by the vector field $X_M$.
The odd operator $D=d-\iota(X_M)$ is well defined on $\CA_G(\g,M)$ and $D^2=0$.
We define $\CH_G(\g,M)$ as the $\Z/2\Z$ graded cohomology space of $D$.
We often write an element $\alpha\in \CH_G(\g,M)$  as
$\alpha(X)$, meaning a representant $X\mapsto \alpha(X)$
 such that
$(d-\iota(X_M))\alpha(X)=0$, and we will say  that $\alpha$ is a function of $X\in \g$.
There is a natural map $\CH_G(\g,M)\to H^*(M)$ given by
$\alpha\mapsto
\alpha(0)$.

Let $S\in \g$, we have a map
$\tau_S:\CA_G(\g,M)\to \CA_G(\g,M(S))$
defined by $\tau_S\alpha(X)=\alpha(S+X)|_{M(S)}$.
Since $S_M$ vanishes on $M(S)$, this operator commute with the differentials and induce an application $\tau_S: \CH_G(\g,M)\to \CH_G(\g,M(S)).$

  We choose a scalar product on $\g$, and, for $a>0$, let
    $B_a\subset \g$ be the open ball of radius $a$. Define the space
    $\CA_G(B_a,M)$  as the space  of $G$-equivariant forms $X\mapsto \alpha(X)$ defined on $B_a$ and  $\CH_G(B_a,M)$ as the $\Z/2\Z$ graded cohomology space of $D$ on $\CA_G(B_a,M)$.
If $b\leq a$, there is a natural restriction map $\CH_G(B_a,M)\to \CH_G(B_b,M)$.
This allows us to define $$H_{[G]}(M)=\lim_{a\to 0} \CH_G(B_a,M),$$
the space of germs at $0$ of equivariant cohomology classes (with analytic coefficients).

If $\alpha(0)\in H^*(M)$
is invertible, then $\alpha^{-1}(X):=\frac{1}{\alpha(X)}$
is in $H_{[G]}(M)$.

{\bf If $\alpha(X)\in H_{[G]}(M)$, and $M$ is oriented, then
$X\mapsto \int_M \alpha(X)$ is a germ at $0$ of an analytic function
on $\g$.}

\bigskip

\bigskip
We now define the sheaf $\CSH_{\g}(M)$ on $\g$ for equivariant cohomology.
If $S\in \g$, the stalk is the
space $H_{[G]}(M(S))$ of germs of $G$-equivariant classes on $M(S)$.
\begin{definition}\label{def:sheaf}
A collection ${\bf b}=(b_S)_{S\in U}$ of elements
$b_S\in H_{[G]}(M(S))$ is a section of $\CSH_{\g}(M)$ over $U$
if and only if, given $S\in U$, there exists $a>0$  sufficiently small
such that

1) $S+B_a\subset U$

2) $M(S+Y)=M(S)\cap M(Y)$ for all $Y\in B_a$.

3) There exists a representative of $b_S$ defined on $B_a$
 such that  the equality $$b_S(Y+X)|_{M(S+Y)}=b_{S+Y}(X)$$
holds  as elements  of $H_{[G]}(M(S+Y))$ for every $Y\in B_a$.

\end{definition}

We write informally:
$$b_S(Y+X)|_{M(S+Y)}=b_{S+Y}(X)$$
for $S \in U$,  $Y$ sufficiently small.
This is an equality  of germs of  functions of $X$, that is an equality in $H_{[G]}(M(S+Y))$.

If $M$ is a point, the sheaf $\CSH_{\g}(M)$ is the sheaf of analytic functions on $\g$.

\bigskip

Similarly we define a sheaf $\CSK_G(M)$ on $G$ for equivariant $K$-theory.
If $g\in G$, the stalk at $g$ is the space $H_{[G]}(M(g))$
of germs of $G$-equivariant classes on $M(g)$.
A collection ${\bf \alpha}=(\alpha_g)_{g\in U}$  of elements   $\alpha_g\in H_{[G]}(M(g))$ is a section of $\CSK_G(M)$ over $U$ if and only if (with formal definition as  in Definition \ref{def:sheaf})
we have $$\alpha_g(Y+ X)|_{M(g \exp Y)}=\alpha_{g\exp Y}(X)$$
for $g \in U$, $Y$ sufficiently small so that $g\exp Y\in U$,
and $M(g\exp Y)=M(g)\cap M(Y)$.
This is an equality in $H_{[G]}(M(g\exp Y))$.
Such a collection $(\alpha_g)$ will be called a bouquet of equivariant forms.

If $M$ is a point, $\CSK_G(M)$ is the sheaf of analytic functions on $G$.

\bigskip

Let us recall the Chern Weil construction of  equivariant characteristic classes.

We will say that a vector bundle $\CW\to M$ has typical fiber   a (real or complex)  vector space $W$ and structure group
$H$ if locally
$\CW=U\times W$ with transition functions $g_{ij}$ over $U_i\cap U_j$
valued in a group $H$ provided with an homomorphism
 $H\to GL(W)$. One of the relevant example in this note
 is   a vector bundle $\CW$ with spin structure.
So here $N$ is an oriented even dimensional Euclidean space and
$H={\rm Spin}(N)$, a double  cover of  $SO(N)$.

If $s\in GL(W)$, we denote by $GL(W,s)$ the group of transformation of $W$ commuting
with $s$, with Lie algebra $\gl(W,s)$.
 If $\CW$ is a $G$-equivariant vector bundle over a connected manifold $M$, and $g\in G$ acts trivially on $M$, then  we can choose local trivialization $U\times W$ so that the vertical action of $g$ on  $ W$ is  constant (since $G$ is compact), so the structure group $H$ of $\CW$  can be reduced to  $GL(W,g)$, where $GL(W,g)$ is the subgroup of transformations of $W$ commuting with $g$.

Let $\CN\to M$ be a $G$-equivariant real vector bundle of rank $r$.
Choosing a $G$-invariant connection $\nabla$,
one defines the equivariant curvature ${\mathcal R}(X)$ of $\CN$
(see  \cite{BV82}, see also \cite{BGV}, Chapter 7 ). Then, for $X\in \g$,
${\mathcal R}(X)=\mu(X)+\Omega$ is a
${\rm End}(\CN)$ valued differential form on $M$.
Here $\Omega=\nabla^2$ is the usual curvature  of $\nabla$ (a two-form valued in ${\rm End}(\CN))$,
and $\mu(X)$ a function on $M$ with value in ${\rm End}(\CN)$.
 If $S_M$ vanishes on $M$, we will repeatedly use the
 property
 \begin{equation}\label{eq:RS}
   {\mathcal R}(S+X)=\CL(S)+{\mathcal R}(X)
   \end{equation}
    where $\CL(S)$ denotes the endomorphism of $\CN$ produced by the action of $S$ on the fibers of $\CN$.

If $R\mapsto \Phi(R)$ is a $GL(r)$ invariant polynomial  function on $\gl(r)$,
the Chern Weil homomorphism associate to the polynomial function $\Phi$ an equivariant  characteristic class denoted by
${\rm cw}(\Phi)(X)$  on $M$. This is the class of the equivariant closed form
$X\mapsto \Phi({\mathcal R}(X))$, that is we replace the scalar valued matrix $R$ by $\CR(X)$ which is a matrix with coefficients even differential forms on $M$.

Let $\Phi$ be a germ (at $0$) of a $GL(r)$ invariant analytic function on
$\gl(r)$. Since the dependance of $\mu(X)$ in $X$ in linear and $\Omega$ nilpotent,
(and $M$ compact), it is easy to see that
 ${\rm cw}(\Phi)(X)$ is  in $H_{[G]}(M)$, that is, depends  analytically  of $X$ near $X=0$.
  Similarly let $H\to GL(W)$ be a Lie  group with Lie algebra $\frak h$.
If $\Phi$ is a germ (at $0$) of  a $H$-invariant analytic function on ${\frak h}$,
and $\CW\to M$ a $G$-equivariant vector bundle with typical fiber $W$ and structure group $H$,
 then ${\rm cw}(\Phi)(X)$ is  in $H_{[G]}(M)$.

Let  $N$ be an  Euclidean vector  space.
The function $R\mapsto \det_N(R)$ is a $SO(N)$ invariant polynomial on $\so(N)$ (identically $0$, unless $\dim N$ is even).
 An orientation $o_N$ on $N$ determines a square root
 $\det_{N,o_N}^{1/2}(R)={\rm Pfaff}(R)$.
 Conversely, if $S\in \so(N)$ is invertible,
 $S$ determines an orientation $o_S$ on $N$ such that  $\det_{N,o_S}^{1/2}(S)>0$.

Let $N=\R e_1\oplus \R e_2$, with orthonormal basis $e_1,e_2$
and  orientation given by $e_1\wedge e_2$.
Let $R\in \so(N)$ given by  $R e_1=\theta  e_2, Re_2=-\theta e_1$,
then $\det_{N,o_N}^{1/2}(R)=\theta$.

\bigskip

Let $M$ be a   $G$ manifold and $\CN\to M$ be a $G$-equivariant oriented even dimensional real vector bundle with typical fiber $N$ and orientation $o_\CN$.
Choose a $G$ invariant Euclidean structure on $\CN$.
The Euler class ${\rm Eul}(\CN,o_\CN)(X)$ of the bundle $\CN\to M$ is the image by the Chern Weil homomorphism of the function $R\mapsto \det_{N,o_N}^{1/2}(R)$.
So  ${\rm Eul}(\CN,o_\CN)(X)=\det_{\CN,o_\CN}^{1/2}({\mathcal R}(X))$ where ${\mathcal R}(X)$ is the equivariant curvature of an invariant Euclidean connection.
Its equivariant cohomology class does not depend of the choice of Euclidean structure on $\CN$ nor of a choice of Euclidean connection.
 Let $S\in \g$ such that $S_M=0$.
 Then ${\rm Eul}(\CN,o_\CN)(S+X)=\det_{\CN,o_\CN}^{1/2}(\CL(S)+{\mathcal R}(X))$.
  If the transformation $\CL(S)$ produced by $S$ on the fibers of $\CN$  is invertible,
the equivariant cohomology class $X\mapsto {\rm Eul}(\CN,o_\CN)(S+X)$ is invertible at $X=0$.
Let us  recall the BV localization formula
\cite{BV82}.
If $M$ is a compact oriented manifold (not necessarily connected),
for $S\in \g$,  the normal bundle $\CN=TM/TM(S)$ of $M(S)$ in $M$
is an oriented even dimensional vector bundle
with orientation $o_S$ associated to the invertible action of $S$ on fibers.
The manifold $M(S)$ might not be connected. Each connected component $C_S$
of $M(S)$ is contained in a connected component $C$ of $M$ and $\dim(C)-\dim(C_S)$
is even valued. This provides a locally constant function  $(\dim M -\dim M(S))/2$ on $M(S)$.
The manifold $M(S)$ can be oriented by the quotient orientation of $M$ and $(\CN,o_S)$
(here and later, since $\CN$ is even dimensional, we do not care on order
defining quotient orientations).

If $o_M$ is the orientation of $M$, and $o_{M(S)}$ an orientation of $M(S)$,
we denote by $o_M/o_{M(S)}$ the corresponding orientation of $TM/TM(S)$.
If $\alpha$ is a germ of equivariant cohomology class at $0$,
then $\int_M\alpha(X)$ is a germ of analytic function at $X=0$.
 Furthermore,
$$\int_M\alpha(S+X)=\int_{M(S)}
(-2\pi)^{(\dim M-\dim M(S))/2} \frac{\alpha(S+X)}
{{\rm Eul}(TM/TM(S),o_M/o_{M(S)})(S+X)}$$
for $X$ small enough. Integrations on $M$, $M(S)$ are performed with the orientations $o_M$, $o_{M(S)}$, respectively.

If $M$ is compact oriented,
 and $\alpha$ is an equivariant closed class with polynomial coefficients, then
 $X\mapsto \int_M \alpha(X)$ is a polynomial function on $\g$ which depends only of the cohomology class of $\alpha$.
{\bf We extend this integration to sections of the sheaf $\CSH(\g,M)$}: let  $\alpha=(\alpha_S)_{S\in U}$ be a section
 of the sheaf $\CSH_{\g}(M)$ over an open set $U\subset \g$.
 For each $S\in U$, we define a germ $I_S$ of analytic function at $S$ by
 $$I_S(S+X)=\int_{M(S)} (-2\pi)^{-\dim M(S)/2}
 \frac{\alpha_S(X)}{{\rm Eul}(TM/TM(S),o_{M}/o_{M(S)})(S+X)}$$
 for $X$ small enough.
  Here and later, $M(S)$ might not be connected, so $\dim M(S)$
 is a locally constant function on $M(S)$.
 The orientation on $M(S)$ is deduced from the orientation on $M$ and the orientation $o_S$ of $TM/TM(S)$.

So $I_S$ is a germ in $S\in U$ of an analytic function.
By definition of the sheaf $\CSH_{\g}(M)$ and the BV localization formula,
we see that there exists a function $I$ on $U$ such that
 $I_S(S+X)=I(S+X)$. Indeed, for $Y$
small, the BV localization formula implies
$I_S(S+(Y+X))=I_{S+Y}((S+Y)+X)$, when $X$ is small enough.

\bigskip

To define $K$-integration on $\CSK_G(M)$, we need the equivariant $\hat A$ class.
So from now on, $M$ will be an even dimensional oriented spin manifold.

Let  $V$ be a real vector space. Let $R\in \gl(V)$.
Then $$j(V)(R)=\det_V(\frac{e^{R/2}-e^{-R/2}}{R})$$ is a $GL(V)$ invariant function,
 such that $j(V)(0)=1$.
So $j^{-1/2}(V)$   defines a germ  of analytic function on $\gl(V)$,
with $j^{-1/2}(V)(0)=1$.
When $R$ is trace less, $j(V)(R)$ is also equal to
$\det_V(\frac{e^{R}-1}{R}).$

\begin{definition}
If $\CV$ is a real vector bundle on $M$ with typical fiber $V$,
 the equivariant  $\hat A$ class on $\CV$ is denoted by $\hat A(\CV)(X)$
 and is the image by the Chern Weil map of $j^{-1/2}(V)$.
 If $\CV=TM$ is the tangent bundle, we denote this equivariant class by $\hat A(M)(X)$.
\end{definition}
So $X\mapsto \hat A(\CV)(X)$ is an element of $H_{[G]}(M)$.

 The $K$-theoretical integration of a bouquet ${\bf b}=(b_g)_{g\in U}$, defined on a small neighborhood $U$ of $1\in G$,
 is the germ at $1\in G$ of the analytic function  $\theta$ on $G$ such that for $X\in \g$, small enough,
 $$\theta(\exp X)=\int_M (-2i\pi)^{-(\dim M)/2} b_1(X) {\hat A}(M)(X).$$

This formula suggests the $K$-integration formula everywhere.
However, we need that $M(g)$ is oriented for any $g\in G$.
When $M$ has a $G$ invariant spin structure, this is the case.
In the next section, we recall some facts on the spinor group,
 in particular we explain why $M(g)$ is oriented.

\section{The spinor group and some invariant functions}\label{sec:spinors}
We follow the notations of    \cite{BGV}, chapter 6.

Let $N$ be an even  dimensional Euclidean  vector space.
We denote by ${\rm Clifford}(N)$ the Clifford algebra of $N$.
When $N=\{0\}$, then ${\rm Clifford}(N)=\C$.
If $N=N_0\oplus N_1$ is the direct sum of two Euclidean spaces, then
${\rm Clifford}(N)$ is the tensor product of ${\rm Clifford}(N_0)$ and ${\rm Clifford}(N_1)$.
We consider ${\rm Spin}(N)\subset {\rm Clifford}(N)$.
This is a connected compact Lie group, except when $N=\{0\}$, where ${\rm Spin}(N)=\{\pm 1\}$.
In the subsequent definitions, we avoid the case where $N=\{0\}$, where
 definitions have to be adapted!.

 There is a symbol map isomorphism $\sigma: {\rm Clifford}(N)\to \Lambda N$.
Let $g\in {\rm Spin}(N)$.
Then $g$ acts on $N$ by a special orthogonal transformation still denoted by $g$.
The map ${\rm Spin}(N)\to SO(N)$ is a covering of order $2$.
This map induces an isomorphism of the Lie algebra of ${\rm Spin}(N)$ with
the Lie algebra $\so(N)$ of $SO(N)$.
So for $X\in \so(N)$, we can  compute $\exp (X)$ either in $SO(N)$ or in
${\rm Spin(N)}$. We will indicate where we compute this exponential when
the context is ambiguous.

Assume that $N$ is even dimensional.
An orientation $o_N$ on $N$ determines a decomposition of the spinor space $S_N$ of $N$ in $S_N^+\oplus S_N^-$.
Denote
 $${\rm Tr}(g,S_N)={\rm Tr}(g,S_N^+)+{\rm Tr}(g,S_N^-),$$
$${\rm Str}(g,S_N,o_N)={\rm Tr}(g,S_N^+)-{\rm Tr}(g,S_N^-).$$

Let $N=\R e_1\oplus \R e_2$, where $(e_1,e_2)$ is an orthonormal basis
and orientation $o_N$ given by $e_1\wedge e_2$.
Then in ${\rm Clifford}(N)$, $e_1^2=-1,e_2^2=-1, e_1 e_2+e_2 e_1=0$.
So $g=\exp(\theta e_1 e_2)=\cos(\theta)+\sin(\theta) e_1e_2$ in
${\rm Spin}(N)$
and

\begin{equation}\label{sigmag}
\sigma(g)=\cos(\theta)+\sin(\theta) e_1\wedge e_2.\end{equation}

\begin{equation}\label{Tr}
{\rm Tr}(g,S_N)=e^{-i\theta}+e^{i\theta}.\end{equation}

\begin{equation}\label{Str}
{\rm Str}(g,S_N,o_N)=e^{-i\theta}-e^{i\theta}.\end{equation}

Let $X\in \so(N)$ given by $Xe_1=\phi e_2$, $Xe_2=-\phi e_1$.
The element $X$ is identified to $\frac{\phi}{2} e_1 e_2$ in
the Lie algebra of ${\rm Spin}(N)\subset {\rm Clifford} (N)$.
The exponential
$\exp (X)$ in $SO(N)$ is
$$\left(
  \begin{array}{cc}
    \cos(\phi) & -\sin(\phi) \\
    \sin(\phi) & \cos(\phi) \\
  \end{array}
\right)$$
while the exponential of $X$ in ${\rm Spin}(N)$
is
$\cos(\phi/2)+\sin(\phi/2)e_1e_2$.

\begin{definition}\label{def:defchiK}
Let $(N,o_N)$ be an oriented  even dimensional Euclidean  vector space.
  Define the function $\chi(N,o_N)$ on $\so(N)$  by
  $$\chi(N,o_N)(R)=\frac{1}{{\rm Str}(e^R,S_N,o_N)}.$$
  \end{definition}
The function $R\mapsto \chi(N,o_N)(R)$ is defined and analytic on the open set
where  ${\rm Str}(e^R,S_N,o_N)\neq 0$.
In any case,
 $\det^{1/2}_{N,o_N}(R)\chi(N,o_N)(R)$ is analytic at $R=0$,
and (as verified for $\dim N=2$)

\begin{equation}\label{eq:jchi}
j^{-1/2}(N)(R)\det^{-1/2}_{N,o_N}(R)=(-i)^{\dim N/2}\chi(N,o_N)(R).
\end{equation}

{\bf Assume $({\rm Id}-g)$ is an invertible transformation on $N$.
Then the top term of $\sigma(g)\in \Lambda N$ is not zero,
 so defines an orientation $o_g$ on $N$.}
\begin{definition}\label{def:Ospin}
If $Y\in \so(N)$ is such that  $\det_N({\rm Id}_N-e^Y)\neq 0$, we denote by
 ${\rm Os}(Y,o_N)$ the ratio of the orientation provided with $\exp(Y)$
 (computed in ${\rm Spin}(N)$) and a given orientation $o_N$ of $N$.
 \end{definition}

If $Y\in \so(N)$ is invertible and sufficiently small, the orientation determined by
$Y$ or by $\exp(Y)$ (in ${\rm Spin}(N)$)  are the same.

\begin{definition}
Let $(N,o_N)$ be an even dimensional  Euclidean vector  space, and $g\in {\rm Spin}(N)$.
We define a function $\chi(g;N,o_N)$ on $\so(N,g)$ by
$$\chi(g;N,o_N)(R)=\frac{1}{{\rm Str}(ge^R,S_N,o_N)}.$$
This is an $SO(N,g)$ invariant function of $R\in \so(N,g)$  defined  on the set of $R$ such that
$\det_N({\rm Id}-g e^R)\neq 0$.
If $({\rm Id}-g)$ is invertible, then
$\chi(g; N,o_N)(R)$ defines a germ at $0$
of an
analytic function on $\so(N,g)$.

\end{definition}
Beware that the function $\chi(g;N,o_N)$ depends now of $3$ arguments.
The preceding function
$\chi(N,o_N)$ is the specialization of $\chi(g;N,o_N)$ at $g=1$, that is  $\chi(N,o_N)(R)=\chi(1;N,o_N)(R)$.
More generally, if $R$ commutes with $Y$, we have, for $Y$ small,
$$\chi(N,o_N)(Y+R)=\chi(e^Y;N,o_N)(R).$$
So we have:
\begin{equation}\label{eq:jeul}
j^{-1/2}(N)(Y+R)\det_{N,o_N}^{-1/2}(Y+R)=(-i)^{\dim N/2}\chi(e^Y;N,o_N)(R).
\end{equation}

If   $N$  is an  even dimensional Euclidean vector   space and $g\in {\rm Spin}(N)$,
 we write $$N=N_0\oplus N_1$$
  the $g$-invariant decomposition of $N$  such that  $g$ acts by
the identity on $N_0$, and $({\rm Id}-g)$ restricts to an invertible transformation of $N_1$.
The element $g$ belongs to ${\rm Spin}(N_1)\subset {\rm Spin}(N)$, since $g$ commutes with
$N_0\subset {\rm Clifford}(N)$.
Let $Y\in \so(N)$ commuting with $g$.
We assume that the restriction of $Y$ to $N_0$ is invertible.
If $Y$ is sufficiently small, the transformation $g\exp Y$ of $N$  is invertible.
Choose orientations $o_{N_0},o_{N_1},o_N$ on $N_0,N_1,N$.

\begin{proposition}\label{pro:MAINformulaK}
Let  $R\in \so(N)$ commuting with $g$ and $Y$.
Then, if $Y$ is sufficiently small,
$$\chi(N_0,o_{N_0})(Y+R)\chi(g;N_1,o_{N_1})(Y+R)=\frac{o_N}{o_{N_0}\wedge o_{N_1}}
\chi(g \exp Y ; N,o_N)(R).$$
\end{proposition}

In this formula, if $A\in \so(N)$ is a matrix respecting the decomposition
$N=N_0\oplus N_{1}$,
 $\chi(N_0,o_{N_0})(A)$ denotes $\chi(N_0,o_{N_0})(A_0)$,
 where $A_0$ is the restriction of $A$ to $N_0$, etc...
\begin{remark}\label{rem:orgY}
When $Y$ is sufficiently small and invertible on $N_0$,
${\rm Id}-g \exp Y$ is an invertible transformation of $N$,
and the orientation on $N$ given by $g\exp Y$
coincide with the orientation compatible with the orientation of
$N_1$ given by $g_1$ and the orientation on $N_0$ given by $Y$.
\end{remark}

\begin{proof}
 The left hand side is
$$\frac{1}{{\rm Str}(e^{Y+R},S_{N_0},o_{N_0})}\frac{1}{{\rm Str}(ge^{Y+R},S_{N_1},o_{N_1})}.$$

We have $S_N=S_{N_0}\otimes S_{N_1}$ and the transformation
$g e^{Y+R}$ of $S_N$ is the tensor product of the transformations
$e^{Y+R}|_{S_{N_0}}$ and  $g_1e^{Y+R}|_{S_{N_1}}$.
Use multiplicativity of supertraces, this is
$\chi(g;N,o_{N_0}\wedge_{N_1})(Y+R).$
Since $R$ commutes with $Y$,
this is equal to
$$\chi(ge^Y;N,o_{N_0}\wedge o_{N_1})(R).$$
\end{proof}

\section{Equivariant characteristic classes and integration}

\begin{definition}
Let $M$ be a  manifold with a $G$ action.
Let $\CN\to M$ be an  oriented even dimensional Euclidean vector bundle provided
with an equivariant spin structure and orientation $o_\CN$, and typical fiber $N$.
 Let $g\in G$ {\bf acting trivially} on $M$ and such that $1-g$ determines
  an invertible transformation of $\CN$.
We denote by ${\rm cw}\chi(g;\CN,o_\CN)(X)$
the equivariant characteristic class associated by the Chern-Weil homomorphism
to the function $\chi(g;N,o_N)$ on $\so(N,g)$, that is
$${\rm cw}\chi(g;\CN,o_\CN)(X)=\frac{1}{{\rm Str}( g e^{{\mathcal R}(X)},S_\CN,o_\CN)}.$$
\end{definition}

Let $M$ be an even dimensional   spin  manifold with a $G$-action and orientation $o_M$.
Consider $g\in G$, $m\in M(g)$,  and $\CN=TM/TM(g)\to M(g)$.
 So $g$ acts trivially on $M(g)$, and it produces an element $g\in {\rm Spin}(\CN_m)$
 with non zero top symbol term.  We denote by $o_g$  the orientation of $\CN$ produced by $g$.
 {\bf Then  $M(g)$  is even dimensional and
 is oriented by the quotient orientation}.

\begin{definition}
Let $M$ be a compact oriented even dimensional manifold with a $G$ invariant spin structure.
We denote by $\nu_g(X)\in H_{[G]}(M(g))$ the germ of equivariant class on
 $M(g)$ given by:
$$\nu_g(X)=(-2i\pi)^{-\dim M(g)/2}{\hat A}(M(g))(X) {\rm cw}\chi(g;\CN,o_g)(X).$$
Here $\CN=TM/TM(g)$ is the normal bundle of $M(g)$ in $M$.
\end{definition}

The following theorem is a particular case of the push-forward formula in
\cite{duflo-vergne} (see Theorem 134).

\begin{theorem}\label{theo:Kint}
Let ${\bf b}=(b_g)_{g\in U}$ be a section of the sheaf $\CSK_G(M)$ defined over
$U\subset G$.  Then the germs of functions $\theta_g$ defined at $g\in U$
by
$$\theta_g(g \exp X)=\int_{M(g)} b_g(X) \nu_g(X)$$
glue to an analytic function $\theta$ on $U$.
\end{theorem}

 This means that for $Y$ sufficiently small, we have the consistency relation
 $$\theta_g(g (\exp Y \exp X))=\theta_g(g \exp (Y+ X))=
 \theta_{g\exp Y}((g \exp Y) \exp X)$$
 as equality of germs at $0$ of analytic functions of $X\in \g$.  Thus, this insures
 that we can find  an analytic function $\theta$ on $U$ so that
 $\theta_g(g\exp X)=\theta(g\exp X)$.

\begin{proof}
Fix $Y$ sufficiently small so  that $ M(g\exp Y)=M(g)\cap M(Y)$.
Let $$I_g(X)=\int_{M(g)} b_g(X) \nu_g(X).$$
 We have to prove that
 $I_g(Y+X)=I_{g \exp Y}(X)$ as an equality of germs of analytic functions at $X=0$.

Let $\CN_1\to M(g)$  be the normal bundle of $M(g)$ in $M$ oriented by $g$,
and $\CN\to M(g e^Y)$ be the normal bundle of $M(ge^Y)$ in $M$ oriented by $ge^Y$.

Decompose
$$TM|_{M(g \exp Y)}={\mathcal T}_0\oplus \CN_{0}\oplus \CN_1|_{M(g)\cap M(Y)},$$
where ${\mathcal T}_0$ is the tangent bundle $T(M(g\exp Y))$,
$\CN_{0}$ the normal bundle of $M(g\exp Y)=M(g)\cap M(Y)$ in $M(g)$, oriented by
$Y$.

From the localization formula and the definition of the sheaf $\CSK_G(M)$,
it is sufficient to prove
that
$$(-2\pi)^{rk \CN_0/2}\nu_g(Y+X)|_{M(g)\cap M(Y)}
{\rm Eul}^{-1}(\CN_0,o_Y)(Y+X)=\nu_{g\exp Y}(X)$$
as an identity in $H_{[G]}(M(g)\cap M(Y))$.
So we need to prove that
$$\left({\hat A}(M(g))(Y+X)
{\rm cw}\chi(g;\CN_1,o_g)(Y+X)\right)|_{M(g)\cap M(Y)}{\rm Eul}^{-1}(\CN_0,o_Y)(Y+X)$$
is equal to
$$(-i)^{\dim \CN_0/2}{\hat A}(M(g\exp Y))(X)
{\rm cw}\chi(g\exp Y;\CN,o_{g\exp Y})(X).$$

Since $Y$ vanishes on $M(g)\cap M(Y)$, the  class
${\hat A}(M(g))(Y+X)$ restricted to $M(g\exp Y)$
is equal to
${\hat A}(M(g\exp Y) )(X) {\hat A}(\CN_0)(Y+X)$.

Factoring ${\hat A}(M(g\exp Y))$, it is sufficient to prove that

$${\hat A}(\CN_0)(Y+X){\rm Eul}^{-1}(\CN_0,o_Y)(Y+X)
{\rm cw}\chi(g;\CN_1,o_g)(Y+X)|_{M(g)\cap M(Y)}$$
is equal to
$$(-i)^{\dim(\CN_0)/2}{\rm cw}\chi(g\exp Y;\CN,o_{g\exp Y})(X).$$

  Applying the Chern-Weil homomorphism to Formula \ref{eq:jeul},
$$(-i)^{-\dim \CN_0/2}\hat A(\CN_0)(Y+X){\rm Eul}^{-1}(\CN_0,o_Y)(Y+X)=
{\rm cw}\chi(\CN_0,o_Y)(Y+X).$$

So we need to prove that
$${\rm cw}\chi(\CN_0,o_Y)(Y+X)
{\rm cw}\chi(g;\CN_1,o_g)(Y+X)|_{M(g)\cap M(Y)}$$
is equal to
$${\rm cw}\chi(g\exp Y;\CN,o_{g\exp Y})(X).$$

The bundle $\CN_0$ is provided with  a vertical action (locally constant) of $Y$;
So its structure group can be reduced to $SO(N_0,Y)$.
The bundle $\CN_1$ is provided with a vertical action (locally constant) of $g,Y$;
So its structure group is reduced to  $SO(N_1,Y)\cap SO(N_1,g)$.
So their respective
equivariant curvatures are valued in the Lie algebra of these groups.

Since $Y$ vanishes on $M(Y)$, the  various equivariant curvatures
${\mathcal R}(Y+X)$ restricted to $M(Y)$
are $\CL(Y)+{\mathcal R}(X)$, by Formula \ref{eq:RS}.

So the equivariant form
 $$X\mapsto {\rm cw}\chi(\CN_0,o_Y)(Y+X)$$
 is the image by the Chern Weil isomorphism
  of the invariant function
  $$R\mapsto \chi(N_0,o_{Y})(Y+R)$$ on $\so(N_0,Y)$.

 The equivariant form
 $$X\mapsto {\rm cw }\chi(g; \CN_1,o_{g})(Y+X)$$
 is the image by the Chern Weil homomorphism of the invariant function
 $$R\mapsto  \chi(g; N_1,o_{g})(Y+R)$$
 on $SO(N_1,Y)\cap SO(N_1,g)$.

  The equivariant form
 $$X\mapsto {\rm cw}\chi(g\exp Y;\CN,o_{g\exp Y})(X)$$
 is the image by the Chern Weil homomorphism of the invariant function
 $$R\mapsto  \chi(g\exp Y; N,o_{g\exp Y})(R)$$ on $SO(N,g)\cap SO(N,Y)$.

So we see that
 equality follows from Proposition \ref{pro:MAINformulaK}.
\end{proof}

We now define twisted equivariant Chern characters.

\begin{definition}
Let $\CW\to M$ be a $G$-equivariant complex vector bundle with typical fiber $W$.
The equivariant Chern character ${\rm Ch}(\CW)(X)$ of $\CW$  is the image
by the Chern Weil homomorphism of the function $R\mapsto {\rm Tr}(e^R,W)$ on $\gl(W)$.
\end{definition}
Here we consider
$X\mapsto {\rm Ch}(\CW)(X)$ as an element of $H_{[G]}(M)$.

Let $W$ be a complex vector space and $s\in GL(W)$ be a
semi-simple transformation of $W$. Let $\gl(W,s)$ be the space of transformations $R$ of $W$ commuting with $s$.
Then $R\mapsto {\rm Tr}(s e^R,W)$ is an invariant function by $GL(W,s)$.

\begin{definition}
 Let $\CW\to M$ be a $G$ equivariant complex vector bundle over
 $M$ with typical fiber $W$. Assume $g\in G$
acts trivially on $M$.
Then ${\rm Ch}(g,\CW)(X)\in H_{[G]}(M)$,
the $g$-twisted equivariant Chern character,
is the image by the Chern Weil homomorphism of the function
$R\mapsto  {\rm Tr}(g e^R,W)$  on $\gl(W,g)$.
\end{definition}
As usual in the definition above, abusing notations we still denote by $g$ the action of $g$ on $\CW_m$ (the conjugacy class of $g$ is constant on each connected component).

We now define the bouquet of Chern characters.

 Let $\CW\to M$ be a $G$ equivariant complex vector bundle over $M$ with typical fiber $W$.
 Then $g$  acts trivially on $M(g)$.
 So we can define
${\rm Ch}(g;\CW)(X)$, the twisted Chern character, as a $G$ equivariant form on $M(g)$.
It is immediate to see the following proposition.
\begin{proposition}
 ${\bf bch}(\CW)=({\rm Ch}(g;\CW))_{g\in G}$
is a global section of the sheaf $\CSK_G(M)$.
\end{proposition}

Assume $M$ is an even dimensional spin manifold.
Let $D_\CW$ be the Dirac operator of the spin manifold $M$ twisted by $\CW$.
Then the bouquet integral of ${\bf bch}(\CW)$ is an analytic function $\Theta$ on $G$,
and near every $g\in G$, we have the  formula
as a germ of analytic function at $g\in G$:
$$\Theta(g\exp X)=\int_{M(g)}{\rm Ch}(g;\CW)(X)\nu_g(X)$$
and this is the Berline-Vergne formula for the equivariant index of $D_\CW$.
    The $K$-integration formula for general bouquets was defined in order that this formula holds for the bouquet of Chern characters.

In particular, we have
\begin{equation}\label{eq:equidirac}
{\rm index}(D_\CW)(\exp X)=(2i\pi)^{-\dim M/2}\int_M {\rm Ch}(\CW)(X) {\hat A}(M)(X).
\end{equation}

\section{The Witten series}

Let $V$ be a complex vector space.
For $a$ an indeterminate, define
 $\Lambda_a (V)=\oplus_{k=0}^{\dim V} a^k \Lambda^k V$,
$S_a(V)=\oplus_{k=0}^{\infty} a^k S^k V$.

Let $q^{1/2}$ be an indeterminate. Let us consider the Witten series of vector spaces:
$$W_{1,q}(V)=
\prod_{n=1}^{\infty}\Lambda_{q^{n-1/2}}(V) \prod_{n=1}^{\infty}S_{q^{n}}(V)$$
$$ = \C+(V) q^{1/2}+ (V\oplus \Lambda^2V) q+ (V\oplus \otimes^2V
\oplus \Lambda^3V) q^{3/2}+\cdots.$$

If $g$ is a transformation of $V$, it acts on $W_{1,q}(V)$ and one has the equality
$${\rm Tr}(g,W_{1,q}(V))=\frac{\prod_{n=1}^{\infty} \det_V(1+q^{n-1/2}g)}{\prod_{n=1}^{\infty} \det_V(1-q^{n}g)}$$
where the second member is understood as the Taylor series at $q=0$.

The "dimension" of $W_{1,q}(V)$ is given by the trace of $g=1$ and is the series
$$\left(\frac{\prod_{n=1}^{\infty} (1+q^{n-1/2})}{\prod_{n=1}^{\infty} (1-q^{n})}\right)^{\dim V}.$$

\begin{definition}\label{wq}
If $\CV\to M$ is a $G$ equivariant complex vector bundle on $M$, with typical fiber $V$,
 we denote by $W_{1,q}(\CV)$
 the series of bundles with typical fiber $W_{1,q}(V)$.
 \end{definition}

Witten famous rigidity theorem is that the equivariant index
$\Theta(q,g)$
of the Dirac operator $D$ twisted by $W_{1,q}(TM\otimes_\R \C)$
does not depend of $g\in G$.
Following Rosu, we will prove this in the next sections. It is enough to prove it for an $S^1$ action.

Recall that for any non trivial action of $S^1$ on a spin manifold $M$,
${\rm Index}(D)(g)=0$  for any $g\in G$ (this is the Atiyah-Hirzebruch rigidity theorem).
The simplest consequence of Witten rigidity theorem is
that for a spin manifold $M$  with  $S^1$-action,
${\rm Index}(D_\CW)(g)$
is constant for $\CW=TM\otimes_\R \C$, that is ${\rm Index}(D_\CW)(g)={\rm Index}(D_\CW)(1)$, for any $g\in G$.
An example of a non trivial action of $S^1$ on a spin manifold  where ${\rm Index}(D_\CW)(1)$ is not zero is given in \cite{ChenHan} (Example 2).

If we consider the coefficient in $q^{3/2}$ in the series
 $W_{1,q}(TM\otimes_\R \C)$, Witten theorem implies that the sum
 ${\rm Index}(D_{\CW_1})(g)+ {\rm Index}(D_{\CW_2})(g)$
 is a constant function on $G$, when $\CW_1=S^2(TM\otimes_\R \C)$ and $\CW_2=
 \Lambda^3(TM\otimes_\R \C)$. However as seen for $M=P_3(\C)$,
 in general the separate equivariant indices ${\rm Index}(D_{\CW_1})(g)$
 and ${\rm Index}(D_{\CW_2})(g)$ are not constant functions on $G$.

\begin{definition}\label{wq}
If $\CV\to M$ is a $G$ equivariant complex vector bundle on $M$, with typical fiber $V$,
 we denote by ${\rm Ch}(W_{1,q}(\CV))$
 the series of equivariant Chern characters of the series of bundles $W_{1,q}(\CV)$, that is
the image by the Chern-Weil homomorphism
of the function ${\rm Tr}(e^R,W_{1,q}(V)).$
\end{definition}

So the equivariant index of the Dirac operator on $M$ twisted by the series  $W_{1,q}(\CV)$
is the series
\begin{equation}\label{equidiracloop}
{\rm index}(D_{W_{1,q}(\CV)})(q,\exp X)=(2i\pi)^{-\dim M/2}\int_M {\hat A}(M)(X)
{\rm Ch}(W_{1,q}(\CV))(X).
\end{equation}

So Witten theorem is equivalent to the assertion that

$$\int_M \hat A(M)(X){\rm Ch}(W_{1,q}(TM\otimes_\R \C))(X)$$
is independent of $X$.


We will also need the following series:
 $$W_{2,q}(V)=
\prod_{n=1}^{\infty}\Lambda_{-q^{n-1/2}}(V)
\prod_{n=1}^{\infty}S_{-q^{n}}(V)$$
so that
$${\rm Tr}(g,W_{2,q}(V))=\frac{\prod_{n=1}^{\infty}
\det_V(1-q^{n-1/2}g)}{\prod_{n=1}^{\infty} \det_V(1+q^{n}g)}.$$

$$W_{3,q}(V)=
\prod_{n=1}^{\infty}\Lambda_{q^{n}}(V) \prod_{n=1}^{\infty}S_{q^{n-1/2}}(V)$$
so that
$${\rm Tr}(g,W_{3,q}(V))=
\frac{\prod_{n=1}^{\infty}
\det_V(1+q^{n}g)}{\prod_{n=1}^{\infty} \det_V(1-q^{n-1/2}g)},$$
$$W_{4,q}=\prod_{n=1}^{\infty}\Lambda_{-q^{n}}(V)
\prod_{n=1}^{\infty}S_{-q^{n-1/2}}(V)$$
so that
$${\rm Tr}(g,W_{4,q}(V))=\frac{\prod_{n=1}^{\infty} \det_V(1-q^{n}g)}
{\prod_{n=1}^{\infty} \det_V(1+q^{n-1/2}g)}.$$

\section{Equivariant elliptic cohomology}

 Let  $U$ be  now a neighborhood on $0$ in
{\bf the complexified Lie algebra}
$\g_\C$.
Define the space of $G$-equivariant forms
$\CA_{G}(U,M)$ as the space of  maps $Z\mapsto \alpha(Z)$ from $U$ to $\CA(M)^G$
which depends holomorphically  of $Z\in U$.
If $Z=X+\sqrt{-1} Y$ is an element of $U$, with $X,Y\in \g$,
 the operator $\iota(Z_M)$ is defined as
$\iota(X_M)+\sqrt{-1} \iota(Y_M)$ on $\CA(M)$.
If $\alpha(Z)$ satisfies $d(\alpha(Z))-\iota(Z_M)\alpha(Z)=0$, we say that $\alpha$ is
a closed equivariant form with holomorphic coefficients.
We  still denote by $H_{[G]}(M)$ the space of germs of holomorphic equivariant cohomology classes at $Z=0$.
If the action of $G$ on $M$ is trivial, this is simply
$H^*(M)\otimes_{\R}\CO_0(\g_\C)$ where
$\CO_0(\g_\C)$ is the space of germs at $0$ of holomorphic functions on $\g_\C$.

The Chern Weil homomorphism $ \Phi \mapsto {\rm cw}(\Phi)(X)$ can be extended
analytically on $\g_\C$.
So we can define, for $Z\in \g_\C$ small enough,  $ {\hat A}(M)(Z)$ as the equivariant $\hat A$ class, etc....

\bigskip

   We will {\bf from now on} consider the case
   where $G=S^1$ with Lie algebra $\g$.
   We choose $J_{G}\in \g$, a generator of the lattice of elements $X\in \g$ such that $\exp(X)=1$.
The choice of $J_G$ determines an isomorphism of the Lie algebra of $\g$
with $\R$ by $x\mapsto xJ_G$,
and the exponential map $\g\to S^1$ is  $x\mapsto e^{2i\pi x}$.
The complexified Lie algebra $\g_\C$ is isomorphic to $\C$.  The exponential map $\g_\C \to G_\C=\C^*$
is  $z\mapsto e^{2i\pi z}$.

An equivariant form  $\alpha:\g_\C \to \CA(M)$ will be  written as
$z\mapsto \alpha(z)$, where $z\in \C$
is identified with $zJ_G\in \g_\C$, and $\alpha(z)\in \CA(M)$.

We  fix $\tau$  in the upper-half plane.
Let $L_\tau=\Z\oplus \Z \tau $ be the corresponding lattice in $\C$,
and $E_\tau=\C/L_\tau$ be the associated  elliptic curve.
We denote by $\pi_\tau: \C\to E_\tau$ the quotient map.
If $\gamma\in E_\tau$ and $z\in \C$, we still denote by $\gamma+z$ the image of
$\gamma+z$  in $E_\tau$.
We will say that $\gamma \in \C$ is of order $k$ in $E_\tau$ if
$\pi_\tau(\gamma)$ is of order $k$ in $E_\tau$.
So $k\gamma=\alpha+\beta \tau$ with $\alpha,\beta\in \Z$, and if
$a$ is an integer relatively prime with respect to $k$,
then $a\gamma\notin L_\tau$.

We consider the subset $F(M)$ of $S^1$ consisting of the $g$
such that $M(g)\neq M^G$. This is a finite subset of $S^1$ consisting  of elements of finite orders.
We denote by $O(M)$ the set of orders of elements $g\in F(M).$
The set $O(M)$ can be computed as follows.
Consider a connected component $C$ of $M(J_{G})$ and a point $m\in C$. Then the normal bundle $\CN$ of $C$ in $M$
carries an action of $S^1$.
We write $\CN_m\otimes_\R \C=\oplus_{a\in \Delta(C)} N_a$ where, for $v\in N_a$,
 $J_{G}v=2i\pi a v$. Here $\Delta(C)$ is a set  of integers and
 $$O(M)=\cup_{ C}\{a; a>0, a\in \Delta(C)\}.$$
We say that $\gamma\in E_\tau$ is special,
if the order $k$ of $\gamma$ is in $O(M)$. The set of special elements is finite.

Following Rosu, we define a sheaf ${\mathcal Ell}_G(M,\tau)$ over $E_\tau$ as follows.
If  $\gamma\in E_\tau$ is not a special point,
define $M_\gamma=M^G$.
If $\gamma\in E_\tau$ is a special point of order $k$,
define $M_\gamma=M(\exp(J_G/k))$.
Then $M^G\subseteq M_\gamma$ for all $\gamma\in E_\tau$.

The stalk of ${\mathcal Ell}_G(M,\tau)$ at a  point  $\gamma\in E_\tau$
is $H_{[G]}(M_\gamma)$.
The set of special points being finite,
 for any $\gamma\in \C$, then, for $y\in \C$ not $0$ and sufficiently small, $\gamma+y$  is not special.

 By definition (as in Definition \ref{def:sheaf}),
a collection ${\bf b}=(b_\gamma)_{\gamma\in U}$ of elements of $H_{[G]}(M_\gamma)$
is a section of the sheaf ${\mathcal Ell}_G(M,\tau)$
if and only if
$$b_\gamma(y+z)|_{M^G}=b_{\gamma+y}(z)$$
for any $y\neq 0$  in $\C$ small enough.
The equality is an equality of germs at $0$ of equivariant cohomology classes on $M^G$.
Such a collection ${\bf b}=(b_\gamma)_{\gamma\in U}$ will be called an elliptic bouquet of equivariant
forms.
This definition is  a " concrete " version of
 Grojnowski definition
\cite{groj} of delocalized equivariant elliptic cohomology.

If $M$ is a point, the sheaf ${\mathcal Ell}_G(M,\tau)$
is the sheaf ${\mathcal O}(E_\tau)$.

\bigskip

As defined in Rosu, the $\tau$-integration of an elliptic  bouquet defined on $U\subset E_\tau$ is an holomorphic function on $U$.
In particular,  the $\tau$-integration of a global section of ${\mathcal Ell}_G(M,\tau)$
is a constant. The rest of this note is devoted to review, and somewhat simplify, the definition of $\tau$-integration.

\bigskip

If ${\bf b}=(b_\gamma)_{\gamma \in U}$ is defined at $0$, then the local $\tau$-integration of $b$ at $0$ will be  by definition the germ of
the holomorphic function
$$z\mapsto
\int_{M}(-2i\pi)^{\dim M/2}\ b_0(z) {\hat A}(M)(z) {\rm Ch}(W_{1,q}(TM\otimes_\R \C))(z).$$

To define local $\tau$-integration at a general point $\gamma\in E_\tau$,
we will define  equivariant characteristic classes  $\mu_\gamma(z)$  in  Section \ref{mugamma}.
If $b$ is defined at $\gamma\in E_\tau$, then the local $\tau$-integration of $b_\gamma$  will be
the holomorphic function $\theta$ defined in a neighborhood of $\gamma$ by
$$\theta(\gamma+z)= \int_{M_\gamma} b_\gamma(z)
\mu_\gamma(z).$$

Integration at $0$, and the localization formula, suggests the local $\tau$-integration formula at
$\gamma$. The form $\mu_\gamma(z)$ should be (up to a constant)
$${\hat A}(M_\gamma)(z)
{\rm cw}\chi(g;\CN, o_\CN)(z) {\rm Ch}(g; W_{1,q}(TM\otimes_\R \C))(z)$$
with $g=\exp(\gamma J_G)$ in $G_\C$ and
$\CN$ the normal bundle of $M_\gamma$ in $M$.
But, if $\gamma$ is a special element of order $k$,
$\exp(J_G/k)$ generates a
finite subgroup  of $G$ isomorphic to
$\Z/k\Z$, and  we have only a vertical action of $\Z/k\Z$ on the normal bundle $\CN$  and {\bf not} of $J_G$.
So $g=\exp (\gamma J_G)=\exp((k\gamma)(J_G/k))$
in $G_\C$ does not act on $\CN$, except
if $k\gamma\in \Z$, that is $g\in G\subset G_\C$.
Then
$${\rm cw}\chi(g;\CN,o_\CN)(z) {\rm Ch}(g; W_{1,q}(TM\otimes_\R \C))(z)$$
is the image by the Chern-Weil map
of $$R\mapsto \frac{1}{{\rm Str}(ge^R,S_N,o_N)}
 {\rm Tr}(ge^R,W_{1,q}(N\otimes_\R \C)).$$
The main theorem of the next section is Theorem \ref{theo:wantindJ}, where we show that this formula still makes sense when $g=\exp (\gamma J_G)$
is in $G_\C$.

\section{Invariant functions  associated to elliptic functions}

Let $\tau$ be fixed in the upper half-plane and  $q=\exp(2i\pi \tau)$.
So $|q|<1$.

Let
$z\in \C$.
Define the function (the inverse of the Jacobi sine)
$$\Phi_1(z,\tau)=
\frac{1}{(e^{-i\pi z}-e^{i\pi z})}
 \frac{\prod_{n\geq 1}(1+q^{n-1/2}e^{2i\pi z})(1+q^{n-1/2}e^{-2i\pi z})}
 {\prod_{n\geq 1}(1-q^{n}e^{2i\pi z})(1-q^{n}e^{-2i\pi z})}.$$
It satisfies
$$\Phi_1(z+1,\tau)=-\Phi_1(z,\tau)$$
and
$$\Phi_1(z+\tau,\tau)=-\Phi_1(z,\tau).$$
The set of poles of the function $\Phi_1(z,\tau)$ is the lattice $L_\tau$.

Let $$\Phi_2(z,\tau)=\frac{1}{(e^{i\pi z}+e^{-i\pi z})}
\frac{\prod_{n\geq 1}(1-q^{n-1/2}e^{2i\pi z})(1-q^{n-1/2}e^{-2i\pi z})
}{\prod_{n\geq 1}(1+q^{n}e^{2i\pi z}) (1+q^{n}e^{-2i\pi z})},
$$

 $$\Phi_3(z,\tau)=(e^{i\pi z}+e^{-i\pi z})
\frac{\prod_{n\geq 1}(1+q^{n}e^{2i\pi z})(1+q^{n}e^{-2i\pi z})
}{\prod_{n\geq 1}(1-q^{n-1/2}e^{2i\pi z}) (1-q^{n-1/2}e^{-2i\pi z})},
$$

 $$\Phi_4(z,\tau)=(e^{i\pi z}-e^{-i\pi z})
\frac{\prod_{n\geq 1}(1-q^{n}e^{2i\pi z})(1-q^{n}e^{-2i\pi z})
}{\prod_{n\geq 1}(1+q^{n-1/2}e^{2i\pi z}) (1+q^{n-1/2}e^{-2i\pi z})}.
$$

Then $$\Phi_1(z+\frac{1}{2},\tau)=i\Phi_2(z,\tau),$$
 $$\Phi_1(z+\frac{\tau}{2},\tau)= q^{1/4}\Phi_3(z,\tau),$$
 $$\Phi_1(z+\frac{1}{2}+\frac{\tau}{2},\tau)=i q^{1/4}\Phi_4(z,\tau).$$

\bigskip

Let $R\in \gl(V)$,  and $i\in \{1,2,3,4\}$. Then
$R\mapsto {\rm Tr}(e^R,W_{i,q}(V))$
are  germs of analytic functions on $\gl(V)$.
Indeed, for example,
$${\rm Tr}(e^R,W_{1,q}(V))=\frac{\prod_{n=1}^{\infty}
\det_V(1+q^{n-1/2}e^{R})}{\prod_{n=1}^{\infty} \det_V(1-q^{n}e^R)},$$
the  infinite product in the numerator  as well as
the  infinite product in the denominator are convergent for $|q|<1$,
 and the value at $R=0$ is well defined. This is the "dimension" of $W_{1,q}(V)$.

\begin{definition}
For $R\in \so(V)$,  and $i\in \{1,2,3,4\}$
we define $$C_{i}(\tau,V)(R)={\rm Tr}(e^R,W_{i,q}(V\otimes_\R \C)).$$
\end{definition}

\bigskip

\begin{definition}\label{def:chiElliptic}

Let $(N,o_N)$ be an oriented even dimensional Euclidean vector space. Define
the function $Z(\tau,N,o_N)$ on $\so(N)$ by
$$Z(\tau,N,o_N)(R)=\frac{1}{{\rm Str}(e^R,S_N,o_N)}
{\rm Tr}(e^R,W_{1,q}(N\otimes_\R \C)).$$
\end{definition}
So $\det_{N,o_N}^{1/2}(R)Z(\tau,N,o_N)(R)$
is analytic  at $R=0$, and we have
\begin{equation}\label{eq:eqjLchi}
j^{-1/2}(N)(R)  C_1(\tau,N)(R) \det_{N,o_N}^{-1/2}(R)=(-i)^{\dim N/2}
Z(\tau,N,o_N)(R).
\end{equation}

\begin{remark}
If we consider $q$ as a formal parameter, then $Z(\tau,N,o_N)$
coincide for $q=0$ with the function $\chi(N,o_N)(R)=\frac{1}{{\rm Str}(e^R,S_N,o_N)}$
defined in Section \ref{sec:spinors}.
\end{remark}

Let $(N,o_N)$ be an oriented even dimensional Euclidean vector  space, and $J\in \so(N)$ such that
$\exp(J)={\rm Id}_N$.
Break the vector space $N\otimes_\R \C$ in a direct sum
$\oplus_{a\in \Delta} N_a$
with $\dim  N_a=1$. Here $\Delta$ is
a multilist of integers and
 the action of $J$ on $N_a\otimes_\R \C$
is $2i\pi a$. The integers $|a|$ are called the rotation numbers of $J$.
If $J$ is invertible,  $a=0$ is not in $\Delta$.
In this case, the list $\Delta$ has an even number of elements and
if $a\in \Delta$ then $-a\in \Delta$.
We write   $\Delta=\Delta_+\cup \Delta_-$.
Define $\rho(J,\Delta_+)=\sum_{a\in \Delta_+}a$.

Then

\begin{equation}\label{eq:epsiloJ}
\epsilon(J,N)=(-1)^{\rho(J,\Delta_+)}
\end{equation}
does not depend of the choice of $\Delta^+$.

\begin{definition}
Assume that $(N,o_N)$ is an  oriented even dimensional  Euclidean vector space.
Let $J\in \so(N)$, invertible, and such that $\exp(J)={\rm Id}_N$.
Let $R\in \so(N,J)$, and $\gamma\in \C$.
Define
$$Z(\gamma,J;\tau,N,o_N)(R)=\frac{1}{{\rm Str}(e^{\gamma J}e^R,S_N,o_N)}
 {\rm Tr}(e^{\gamma J}e^R,W_{1,q}(N\otimes_\R \C)).$$
\end{definition}

We consider $Z(\gamma,J;\tau,N,o_N)(R)$ as a function of
$R\in \so(N,J)$, defined if $R$ is such that  ${\rm Str}(e^{\gamma J}e^R,S_N,o_N)\neq 0$
and $\det_{N}(1-q^n e^{\gamma J} e^R)$ invertible for all $n\geq 1$.
Here the function $z\mapsto {\rm Str}(e^{zJ}e^R,S_N,o_N)$
is the holomorphic extension of the
analytic function $x\mapsto {\rm Str}(e^{xJ}e^R,S_N,o_N)$
for $x\in \R$.

Beware that now the function $Z(\gamma,J;\tau,N,o_N)$ depends of $5$ arguments.
We have for $R\in \so(N,J)$
 $$Z(\gamma,J;\tau, N,o_N)(R)=Z(\tau,N,o_N)(\gamma J+R).$$

\begin{lemma}
The function
$Z(\gamma,J;N,o_N)(R)$
is invariant by $SO(N,J)$.

Furthermore, if $\gamma\in \C$ is such that $a\gamma\notin L_\tau$ for
all rotation numbers $a$ of $J$,
then the function $R\mapsto Z(\gamma,J;\tau,N,o_N)(R)$
is analytic at  $R=0$.

\end{lemma}

\begin{proof}
Clearly $Z(\gamma,J;\tau,N,o_N)(R)$
is invariant by $SO(N,J)$.
So we need only to consider the case where $R\in \so(N,J)$
is diagonal with respect to the decomposition
$N\otimes_\R \C=\oplus_{a\in \Delta} N_a$.
We write $R v=2i\pi r_a v$ for $v\in N_a$,
so $Rv=-2i\pi r_{a}v$ on $N_{-a}$, since $R\in \so(N,J)$.
Let $\nu(o_N,\Delta^+)\in \{1,-1\}$
such that
$${\rm Str}(\exp(zJ),S_N,o_N)=\nu(o_N,\Delta^+)\prod_{a\in \Delta^+} (e^{-i \pi az}-e^{i\pi az}).$$
So

\begin{equation}\label{eq:FundR}
Z(\gamma,J;\tau,N,o_N)(R)=
\nu(o_N,\Delta^+)\prod_{a\in \Delta^+} \Phi_1(a\gamma+r_a,\tau).\end{equation}

Since the set of poles of $z\mapsto \Phi_1(z,\tau)$ is $L_\tau$,
we see that  the function $R\mapsto Z(\gamma,J;\tau,N,o_N)(R)$
(on $\so(N,J)$) is well defined as a germ of analytic function at $R=0$
 when $\gamma\in \C$ is such that
$a\gamma\notin L_\tau$ for all rotation numbers $a$ of $J$.
\end{proof}

Let us state some properties of the function
$Z(\gamma,J;\tau,N,o_N)$.

\begin{lemma}\label{lem:Zperiodic}
For $\gamma,y\in \C$, $R\in \so(N,J)$,
we have
$$Z(\gamma,J;\tau,N,o_N)(yJ+R)= Z(\gamma+y,J;\tau,N,o_N)(R),$$
$$Z(\gamma+1,J;\tau,N,o_N)(R)= \epsilon(J,N) Z(\gamma,J;\tau,N,o_N)(R),$$
$$Z(\gamma+\tau,J;\tau,N,o_N)(R)= \epsilon(J,N) Z(\gamma,J;\tau,N,o_N)(R).$$

\end{lemma}
\begin{proof}
The first formula is obvious since $e^{\gamma J}e^{yJ+R}=e^{(\gamma+y)J}e^R$
if $R$ commute with $J$.

It is enough to prove the second formula
when $R$ respects the decomposition $N\otimes_\R \C=\oplus_{a\in \Delta^+} N_a$.
We use Formula \ref{eq:FundR}:
$$Z(\gamma+1,J; \tau,N,o_N)(R)=\nu(o_N,\Delta^+)\prod_{a\in \Delta^+}
\Phi_1(a\gamma+r_a+a,\tau)$$
$$=\nu(o_N,\Delta^+)(-1)^{\sum_{\Delta^+} a}\prod_{a\in \Delta^+}
\Phi_1(a\gamma+r_a,\tau)$$
using the periodic behavior with respect to the lattice $L_\tau$ of $\Phi_1$.
The last formula is proved in the same way.
\end{proof}

Assume $J$ is invertible.
 Consider an integer $k$ such that $\exp(J/k)={\rm Id}_N$.
Then the rotation numbers of $J$ are multiple of $k$, so we can define
$\epsilon(J/k,N)=(-1)^{\sum_{\Delta^+}a/k}$.

\begin{proposition}\label{pro:indgammak}
Assume $J$ is invertible.
 Consider an integer $k$ such that $\exp(J/k)={\rm Id}_N$.
Assume $k\gamma=\alpha+\beta\tau $, with $\alpha,\beta\in \Z$.
 Then, for $R\in \so(N,J)$,
  $$Z(\gamma,J;\tau,N,o_N)(R)= (\epsilon(J/k,N))^{\alpha+\beta}Z(\tau,N,o_N)(R).$$
 \end{proposition}
 \begin{proof}
 The rotation numbers of $J$ are $a=kA_a$ with $A_a\in \Z$.
 With same notations as in the preceding proof,
 $Z(\gamma,J;\tau,N,o_N)(R)$ is equal to
 $$\nu(o_N,\Delta^+)\prod_{a\in \Delta^+}
 \Phi_1(a\gamma+r_a,\tau)=\nu(o_N,\Delta^+)\prod_{a\in \Delta^+}
 \Phi_1(r_a+A_a(k\gamma),\tau).$$
 Since $A_a(k\gamma)=A_a \alpha+\tau A_a\beta\in L_\tau$,
 using again the periodic property of $\Phi_1$, we obtain the proposition.
\end{proof}

 We give similar formulae with $\exp(J)={\rm Id}_N$, but $\exp(J/k)=-{\rm Id}_{N}$.
 So $k$ is even.
Since $\exp(J/k)=-{\rm Id}_N$, we can define the sign ${\rm Os}(J/k,o_N)$
(Definition \ref{def:Ospin}).

\begin{proposition}\label{pro:allW}
Let $k$ even, and $\exp(J/k)=-{\rm Id}_{N}$.
Let  $k\gamma=\alpha+\beta \tau$ with $\alpha,\beta\in \Z$.

1) If $\alpha$ and $\beta$  are both even,
 then
 $$Z(\gamma,J;\tau,N,o_N)(R)=c_1(\gamma) ({\rm Os}(J/k,o_N))^{\alpha+\beta}
 \frac{1}{{\rm Str}(e^R,S_N,o_N)}
{\rm Tr}(e^R, W_{1,q}(N\otimes_\R \C))$$
with $c_1(\gamma)=(-1)^{(\alpha+\beta) (\dim N)/4}.$

2)  If $\alpha$ is odd, and $\beta$ even,
 then
 $$Z(\gamma,J;\tau,N,o_N)(R)=c_2(\gamma)({\rm Os}(J/k,o_N))^{\alpha+\beta}
 \frac{1}{{\rm Tr}(e^R,S_N)}
{\rm Tr}(e^R, W_{2,q}(N\otimes_\R \C)),$$
with $c_2(\gamma)=i^{\dim N/2}(-1)^{(\alpha+\beta-1)(\dim N)/4}.$

 3)  If $\alpha$ is even, and $\beta$ odd,
 then
 $$Z(\gamma,J;\tau,N,o_N)(R)=c_3(\gamma)({\rm Os}(J/k,o_N))^{\alpha+\beta}
 {\rm Tr}(e^R,S_N)
{\rm Tr}(e^R, W_{3,q}(N\otimes_\R \C)),$$
with $c_3(\gamma)=q^{\dim N/2}(-1)^{(\alpha+\beta-1) (\dim N)/4}$.

4) If $\alpha,\beta$ are both odd, then
$$Z(\gamma,J;\tau,N,o_N)(R)=c_4(\gamma) ({\rm Os}(J/k,o_N))^{\alpha+\beta}
{\rm Str}(e^R,S_N,o_N)
{\rm Tr}(e^R, W_{4,q}(N\otimes_\R \C)),$$
with $c_4(\gamma)=(iq)^{\dim N/2}(-1)^{(\alpha+\beta) (\dim N)/4}$.

\end{proposition}

When $\alpha+\beta$ is even, we could replace
$({\rm Os}(J/k,o_N))^{\alpha+\beta}$ by $1$, but we prefer to give identical
 formulae for all cases;

\begin{proof}
Write $N\otimes_\R \C=\oplus_{a\in \Delta} N_a$
and $Jv=2i\pi a v$ on $N_a$.
Since $\exp(J/k)=-{\rm Id}_N$,
the rotation numbers $a$ are equal to $k/2$ modulo $k$.
So we write for $a\in \Delta^+$, $a=k/2+k A_a$.
If $N$ is two dimensional with orthonormal basis $e_1, e_2$,
with $Je_1=(k/2+k A)e_2$,
 $Je_2=-(k/2+k A)e_1$, then in
${\rm Spin}(N)\subset {\rm Clifford}(N)$
we have $\exp(J/k)=(-1)^A e_1e_2$.
So $${\rm Os}(J/k,o_N)=(-1)^{\sum_{a\in \Delta^+} A_a}\nu(o_N,\Delta^+).$$

As before,
$$Z(\gamma,J;\tau,N,o_N)(R)=\nu(o_N,\Delta^+)
\prod_{a\in \Delta^+} \Phi_1(a\gamma+r_a,\tau).$$
Now
$a\gamma=k\gamma(1/2+A_a)=(\alpha+\beta \tau)(1/2+A_a).$
We then use the periodicity properties of $\Phi_1,\Phi_2,\Phi_3,\Phi_4$,
with respect to translation by  the lattice $\Z\frac{1}{2}+\Z \frac{\tau}{2}$
and we obtain Proposition \ref{pro:allW}

\end{proof}

An important consequence is the following.
\begin{proposition}
Let $(N,o_N)$ be an oriented even dimensional Euclidean vector  space.
Let $k$ be even and let $\gamma$ be an element of  order $k$ in $E_\tau$.
There exists  an analytic function  $R\mapsto {\rm EM}_{\epsilon}(\gamma;\tau, N,o_N)(R)$
  defined on $\so(N)$ in a neighborhood of $0$,
  {\bf invariant by $SO(N)$},
and such that
{\bf for any $J$ such that $\exp(J/k)=-{\rm Id}_N$},

$$Z(\gamma,J;\tau,N,o_N)(R)=({\rm Os}(J/k,o_N))^{\alpha+\beta}
{\rm EM}_{\epsilon}(\gamma;\tau,N,o_N)(R).$$

\end{proposition}

\begin{proof}
Let  $k\gamma=\alpha+\beta \tau$ with $\alpha,\beta\in \Z$.
If $\gamma$ is exactly of order $k$, the case where
$\alpha,\beta$ are both even does not occur since $k$ is even.

If    $\alpha$ is odd, and $\beta$ even,
define
 $${\rm EM}_\epsilon(\gamma;\tau,N,o_N)(R)=c_2(\gamma)
 \frac{1}{{\rm Tr}(e^R,S_N)}
{\rm Tr}(e^R, W_{2,q}(N\otimes_\R \C)).$$

 If    $\alpha$ is even, and $\beta$ odd,
 define
 $${\rm EM}_\epsilon(\gamma;\tau,N,o_N)(R)=c_3(\gamma){\rm Tr}(e^R,S_N)
{\rm Tr}(e^R, W_{3,q}(N\otimes_\R \C)).$$

If  $\alpha,\beta$ are both odd,
define
 $${\rm EM}_\epsilon(\gamma;\tau,N,o_N)(R)=c_4(\gamma){\rm Str}(e^R,S_N,o_N)
{\rm Tr}(e^R, W_{4,q}(N\otimes_\R \C)).$$

Then  the function $R\mapsto {\rm EM}_\epsilon(\gamma,\tau,N,o_N)(R)$ on $\so(N)$
is invariant by $SO(N)$ and analytic at $R=0$ (since
${\rm Tr}(e^R,S_N)=\dim(S_N)\neq 0$ for $R=0$).
Proposition \ref{pro:allW} implies the first assertion;

\end{proof}

 We now  define
${\rm EM}(\gamma,\zeta;\tau,N,o_N)$ when $\zeta\in SO(N)$
is such that $\zeta^k={\rm Id}_N$, but $\zeta$ does not have the eigenvalues $1$ or  $-1$.

Break $N\otimes_\R \C$ as a direct sum
$\oplus_{u\in F} N_u$
 with $F$
a set of $k$-roots of unity such that
$\zeta v=u v$ when
$v\in N_u$
(here $N_u$ is the eigenspace for the eigenvalue $u$, and might be of dimension $>1$). Consider a decomposition $F=F^+\cup F^-$, with $F^-=\{\overline u; u\in F^+\}$.
We write $F=F^+\cup F^-$ with $F^-=\{\overline{u}, u\in F^+\}$.
Choose $K\in\so(N)$,
 commuting with the action of $\zeta$,
and such that, for $u\in F^+$ and $v\in N_u$,
 $K v=2i\pi a_uv$, $a_u$ integers such that
 $u=e^{2i\pi a_u/k}$.
So $\zeta=\exp(K/k)$ in the group $SO(N)$.
{\bf We will say that such a  $K$ is adapted to $\zeta$}.
From our choice of $K$,  we see that a matrix $R$ commutes with $\zeta$
if and only if $R$ commutes with $K$.
Remark that  the choice of $K$ is not unique: we can change $a_u$ to $a_u+kA_u$
with $A_u$ integer.

\begin{definition}
Assume $(N,o_N)$ is an oriented even dimensional Euclidean vector space.
Let $\zeta\in  SO(N)$
such that $\zeta^k={\rm Id}_N$, and such that  $({\rm Id}\pm \zeta)$ are invertible
transformations of $N$.
Let $\gamma\in \C$ of order $k$ in $E_\tau$, and write
$k\gamma=\alpha+\beta \tau$, with $\alpha,\beta\in \Z$.
Let $J\in \so(N)$ commuting with $\zeta$ and such that $\exp(J/k)=\zeta$ and $R\in \so(N,J)$.
Define
$${\rm EM}(\gamma ,J;\tau,N,o_N)(R)={\rm Os}(J/k,o_N)^{\alpha+\beta}
Z(\gamma,J;\tau,N,o_N)(R).$$
\end{definition}

\begin{theorem}
If $\gamma$ is of order $k$ in $E_\tau$, then

$\bullet$  If $K$ is adapted to $\zeta$, the function
$R\mapsto {\rm EM}(\gamma,K;\tau,N,o_N)(R)$
on $\so(N,\zeta)=\so(N,K)$ defines  a germ at $0$  of an analytic   function on $\so(N,\zeta)$
invariant by the group $SO(N,\zeta)$.

$\bullet$ If $R\in \so(N,\zeta)$,
${\rm EM}(\gamma,K;\tau,N,o_N)(R)$
does not depend of the choice of $K$ adapted to $\zeta$.

$\bullet$
If $K$ is adapted to $\zeta$,
then for any $J$ commuting with $\zeta$ and such that $\exp(J/k)=\zeta$,
we have on $\so(N,J)\subset \so(N,\zeta)$
$${\rm EM}(\gamma,J;\tau,N,o_N)(R)={\rm EM}(\gamma,K;\tau,N,o_N)(R).$$
\end{theorem}

\begin{proof}
First if $\gamma$ is of order  $k$,
we have $a\gamma\notin L_\tau$ for all rotation numbers of $K$.
Indeed $a$ is relatively prime to $k$, since ${\rm I}-\zeta$ is invertible.
So the function $R\mapsto Z(\gamma,K;\tau,N,o_N)(R)$ is well defined at $R=0$.
It is clearly invariant under $SO(N,K)$.
So if $K$ is adapted to $\zeta$, it is invariant by $SO(N,\zeta)$.

Now  take  $J$ commuting with $\zeta$ and such that $\exp(J/k)=\zeta$.
Write
$$Z(\gamma,J;\tau,N,o_N)(R)=
\nu(\Delta^+,o_N)\prod_{a\in \Delta^+} \Phi_1(a\gamma+r_a,\tau).
$$

For $u\in F^+$, write $N_u=\oplus_{a\in \Delta_u^+} N_a$,
and $J v=a v$ on $N_a$.
Then we have $a=a_u+k A_a$, if $a\in \Delta_u^+$,
with $A_a\in \Z$.
Let $\Delta^+=\cup_{u\in F^+}\Delta_u^+$ (as multilists)
so $\rho(K,\Delta^+)-\rho(J,\Delta^+)$ is an integer  divisible by $k$.
We have
$${\rm Os}(J/k,o_N)=
{\rm Os}(K/k,o_N) (-1)^{\frac{1}{k}
(\rho(J,\Delta_+)-\rho(K,\Delta_+))}.$$

Write $a\gamma=a_u\gamma+A_a (\alpha+\beta \tau)$.
Then, using the invariance property of $\Phi_1$, we obtain that
$$Z(\gamma,J;\tau,N,o_N)(R)=(-1)^{(\alpha+\beta)\frac{
\rho(J,\Delta_+)-\rho(K,\Delta_+)}{k}}Z(\gamma,K;\tau,N,o_N)(R).$$
So $${\rm EM}(\gamma,J;\tau,N,o_N)(R)={\rm EM}(\gamma,K;\tau,N,o_N)(R).$$
\end{proof}

\begin{definition}
Assume $(N,o_N)$  is an oriented even dimensional Euclidean vector space.
Let $\zeta\in  SO(N)$
such that $\zeta^k={\rm Id}_N$, and such that ${\rm Id}_N\pm \zeta$ are invertible transformations
 on $N$. Let $\gamma$ be of order $k$ in $E_\tau$.
 We define the function ${\rm EM}(\gamma,\zeta;\tau,N,o_N)$ on $\so(N,\zeta)$ by
 $${\rm EM}(\gamma,\zeta;\tau,N,o_N)(R)={\rm Os}(K/k,o_N)^{\alpha+\beta}
Z(\gamma,K;\tau, N,o_N)(R)$$  for any $K$ adapted to $\zeta$.
\end{definition}

Let us now consider $\zeta\in SO(N)$ such that $\zeta^k={\rm Id}_N$, and $({\rm Id}_N-\zeta)$ invertible.
Write $N=N_B\oplus N_G$,
the decomposition invariant by $\zeta$ and such that
$\zeta$ acts by $-{\rm Id}$ on $N_B$ and ${\rm Id}\pm \zeta$ invertible on
$N_G$.

\begin{definition}
Let $\gamma$ be of order $k$ in $E_\tau$. We define the function ${\rm EM}(\gamma,\zeta;\tau,N,o_N)$ on $\so(N,\zeta)$ by
$${\rm EM}(\gamma,\zeta;\tau,N,o_{N_B}\wedge o_{N_G})=
{\rm EM}_\epsilon(\gamma;\tau,N_B,o_{N_B}) {\rm EM}(\gamma,\zeta;{N_G},o_{N_G}).$$
\end{definition}

The following theorem of independence is a corollary of the preceding discussion.
\begin{theorem}\label{theo:wantindJ}
Let $(N,o_N)$  be an oriented even dimensional Euclidean vector space.
Let $\gamma$ of order $k$ in $E_\tau$ and write
$k\gamma=\alpha+\beta \tau$, with $\alpha,\beta\in \Z$.
Let $\zeta\in  SO(N)$
such that $\zeta^k={\rm Id}_N$ and ${\rm Id}_N-\zeta$ invertible.
Let $J\in \so(N)$ commuting with $\zeta$
and such that $\zeta=\exp(J/k)$.
Then for $R\in \so(N,J)$
$$\frac{1}{{\rm Str}(e^{\gamma J}e^R,S_N,o_N)}
 {\rm Tr}(e^{\gamma J}e^R,W_{1,q}(N\otimes_\R \C))$$
 is equal to
$${\rm Os}(J/k,o_N)^{\alpha+\beta}
{\rm EM}(\gamma,\zeta;\tau,N,o_N)(R).$$
\end{theorem}

Let $\zeta\in SO(N)$ such that $\zeta^k={\rm Id}_N$ and $({\rm Id}_N-\zeta)$ invertible.
Let $o_N$ an orientation on $N$.
We consider the unique $\hat \zeta$ above $\zeta$ in ${\rm Spin}(N)$
and such that the orientation given by $\hat \zeta$ coincide with
$o_N$.
Then $\hat \zeta^k$ is $\pm 1$ in ${\rm Clifford}(N)$.
We define $v(\zeta,k;N,o_N,k)=\hat \zeta^k$.

Then we have the following periodicity property of functions on $\so(N,\zeta)$.
\begin{lemma}\label{lem:gammaplusone}
$${\rm EM}(\gamma+1,\zeta;\tau,N,o_{N})=v(\zeta,k; N,o_N){\rm EM}(\gamma,\zeta;\tau,N,o_N).$$
$${\rm EM}(\gamma+\tau,\zeta;\tau,N,o_{N})=v(\zeta,k; N,o_N)
{\rm EM}(\gamma,\zeta;\tau, N,o_N).$$
\end{lemma}

\begin{proof}
Choosing $J$ commuting with $R$, and such that $\exp(J/k)=\zeta$, we use
$${\rm EM}(\gamma,\zeta;\tau,N,o_N)(R)={\rm Os}(J/k,o_N)^{\alpha+\beta}
Z(\gamma,J;\tau,N,o_N)(R).$$

If $\gamma$ is changed in $\gamma+1$ or in $\gamma+\tau$,
then $\alpha+\beta$ is changed in $\alpha+\beta+k$
and  $Z(\gamma+1;J,\tau,N,o_N)(R)=\epsilon(J,N)Z(\gamma,J;\tau,N,o_N)(R).$

So we have to prove that
${\rm Os}(J/k,o_N)^{k}\epsilon(J,N)=v(\zeta,k;N,o_N)$.
We see that this equality is independent of the choice of $J$.
Now if $N$ is two dimensional with orientation $e_1\wedge e_2$,
we choose $Je_1=(2\pi a) e_2$, $J e_2=-(2\pi a)e_1$, with $a\in [1,2,\ldots,k-1]$.
Then $\hat \zeta=\cos(\pi a/k)+\sin(\pi a/k) e_1e_2$
since $\sin(\pi a/k)>0$, and ${\rm Os}(J/k,o_N)=1$.
 Since $\hat \zeta^k=(-1)^a$, we obtain our lemma.
 \end{proof}

\bigskip

We now prove the analog of Proposition \ref{pro:MAINformulaK} in the elliptic context.

Let   $N$  be an  even dimensional Euclidean vector   space and $\zeta\in SO(N)$,
such that $\zeta^k={\rm Id}_N$.
We write $$N=N_0\oplus N_1$$ the $\zeta$-invariant decomposition
 such that  $\zeta$ acts by
the identity on $N_0$, and $({\rm Id}-\zeta)$ is invertible on $N_1$.

\begin{theorem}(The transfer formula)\label{th:MainformulaElliptic}
Let $k\geq 1$ be an integer.
Let   $N$  be an  even dimensional Euclidean vector   space and let $\zeta\in SO(N)$
such that $\zeta^k={\rm Id}_N$
and let  $\gamma\in \C$ of order $k$ in $E_\tau$.

Let $J\in \so(N)$ commuting with $\zeta$ and such that $\exp(J/k)=\zeta$.
Write $J=J_0\oplus J_1$ in the decomposition $N_0\oplus N_1$.
Choose orientations $o_{N_0},o_{N_1},o_N$ on $N_0,N_1,N$.
Then, for $R\in \so(N,J)$, and $y\in \C$ sufficiently small:
$$Z(\tau,N_0,o_{N_0})(yJ+R){\rm EM}(\gamma,\zeta;\tau,N_1,o_{N_1})(yJ+R)$$
$$=\epsilon \, Z(\gamma+y,J;\tau,N,o_{N})(R)$$
with $$\epsilon={\rm Os}(J_1/k,o_{N_1})^{\alpha+\beta}\epsilon(J_0/k,N_0)^{\alpha+\beta}
o_N/(o_{N_0}o_{N_1}).$$
\end{theorem}

\begin{proof}
Since $\gamma$ is of order $k$, and $\exp(J/k)|_{N_0}={\rm Id}_{N_0}$,
by Proposition \ref{pro:indgammak},
$$Z(\tau,N_0,o_{N_0})(yJ+R)=
\epsilon(J/k,N_0)^{\alpha+\beta}
\frac{{\rm Tr}(e^{\gamma J}e^{yJ+R},W_{1,q}(N_0\otimes_\R \C))}
{{\rm Str}(e^{\gamma J}e^{yJ+R},S_{N_0},o_{N_0})}.$$
By Proposition \ref{theo:wantindJ}
 $${\rm EM}(\gamma,\zeta;\tau,N_1,o_{N_1})(yJ+R)=
 {\rm Os}(J/k,o_{N_1})^{\alpha+\beta}
\frac{{\rm Tr}(e^{\gamma J}e^{yJ+R},W_{1,q}(N_1\otimes_\R \C))}
{{\rm Str}(e^{\gamma J}e^{yJ+R},S_{N_1},o_{N_1})}.$$

By definition $$Z(\gamma+y,J;\tau,N,o_{N})(R)=
\frac{{\rm Tr}(e^{(\gamma+y)J}e^R,W_{1,q}(N\otimes_\R \C))}
{{\rm Str}(e^{(\gamma+y)J}e^R,S_{N},o_{N})}.$$

 As $S_N=S_{N_0}\otimes S_{N_1}$
 and $W_{1,q}(N\otimes_\R \C)=
 W_{1,q}(N_0\otimes_\R \C)\otimes W_{1,q}(N_1\otimes_\R \C),$
 using the multiplicativity of traces and supertraces,
  and the fact that $R$ commutes with $J$,
  we obtain
 the theorem.
 \end{proof}

 Assume now that   the exponential $\hat\zeta=\exp(J/k)$
 in ${\rm Spin}(N)$ is such that $\hat \zeta^k=1$.
 We assume  the infinitesimal action $J$ in $N$ invertible.
 As $\hat \zeta$ projects on $\zeta$,
we see that $\hat \zeta$ belongs to ${\rm Spin}(N_1)\subset {\rm Clifford}(N_1)$,
so defines an orientation on $N_1$.

\begin{theorem}\label{theo:spinequal}
Choose the orientation $o_N$ on $N$ given by $J$,
the orientation on $N_1$ given by $\hat \zeta$, and
the orientation on $N_0$ given by $o_N/o_{\hat \zeta}$.
 Then,  for $R\in \so(N,J)$, and $y\in \C$ sufficiently small,
$$Z(\tau,N_0,o_{N_0})(yJ+R){\rm EM}(\gamma,\zeta;\tau,N_1,o_{N_1})(yJ+R)=
Z(\gamma+y,J;\tau,N,o_{N})(R).$$
\end{theorem}

\begin{proof}
Let $J=J_0\oplus J_1$.
So  $\hat \zeta=\hat\zeta_0 \hat \zeta_1$
with $\hat\zeta_0=\exp(J_0/k)$  and $\hat \zeta_1=\exp(J_1/k).$
We have to verify that
$${\rm Os}(J_1/k,o_{N_1})\epsilon(J_0/k,N_0)=1.$$

We have two cases.

If $\hat \zeta_1=\hat \zeta$, then
 ${\rm Os}(J_1/k,o_{N_1})=1$
 and $\hat \zeta_0=1$.

 In the other case,
  $\hat \zeta_1=-\hat \zeta$, and  ${\hat \zeta}_0=-1$ in ${\rm Spin}(N_0)$.
 All the rotation numbers of $J_0$ are multiples of $k$,
  and in the group ${\rm Spin}(N_0)$
  $$\exp(J_0/k)=\prod_{a\in \Delta^+} (\cos(\pi a/k))=
  (-1)^{\epsilon(J_0/k,N_0)}.$$
So we see that in both cases
 ${\rm Os}(J_1/k,o_{N_1})\epsilon(J_0/k,N_0)=1.$
\end{proof}

In the spin case, we also have the following periodicity properties.

\begin{lemma}
Let $J\in \so(N)$. Assume that, in the group ${\rm Spin}(N)$, we have
$\exp(J)=1$.
Then
$$Z(\gamma+1;\tau,N,o_N)(R)=Z(\gamma;\tau,N,o_N)(R),$$
$$Z(\gamma+\tau;\tau,N,o_N)(R)=Z(\gamma;\tau,N,o_N)(R).$$

\end{lemma}
\begin{proof}
Use Lemma \ref{lem:Zperiodic}. We have $\rho(J,\Delta^+)$ even,
since
the supertrace of $\exp(J)$ in $S_N$ is equal to $0$. So
$\epsilon(J,N)=(-1)^{\rho(J,\Delta^+)}=1$.
\end{proof}

Let $\gamma\in \C$ of order $k$ in $E_\tau$.
Let $\hat \zeta\in {\rm Spin}(N)$ such that ${\hat \zeta}^k=1$,
and orient $N$ by $\hat \zeta$.
Since  $v(\zeta,k;N,o_N)=1$, we obtain the following identities of functions on  $\so(N,\zeta)$.
\begin{proposition}\label{lem:gammaplusonespin}
$${\rm EM}(\gamma+1,\zeta;\tau,N,o_{N})={\rm EM}(\gamma,\zeta;\tau,N,o_N).$$
$${\rm EM}(\gamma+\tau,\zeta;\tau,N,o_{N})={\rm EM}(\gamma,\zeta;\tau,N,o_N).$$
\end{proposition}

So we have accomplished our goal:

A) To $(N,o_N)$ an oriented  even dimensional  Euclidean vector space, $J\in \so(N)$ invertible and
such that $\exp(J)=1$ in ${\rm Spin}(N)$, and $\gamma\in \C$ such that $a\gamma$ is not in $L_\tau$
for all rotation number of $J$,  we have constructed an invariant function
$Z(\gamma,J;\tau,N,o_N)$  on $\so(N,J)$,
depending only of the image of $\gamma$ in $E_\tau$.

A) To $(N,o_N)$ an oriented even dimensional Euclidean vector space,
 $\gamma\in \C$  of order $k$ in $E_\tau$,
 $\zeta\in {\rm Spin}(N)$,
such that $\zeta^k=1$ in ${\rm Spin}(N)$ and such that $({\rm Id}_N-\zeta)$ is invertible,
 we have constructed an invariant function
$EM(\gamma,\zeta; \tau,N,o_\zeta)$  on $\so(N,\zeta)$,
depending only of the image of $\gamma$ in $E_\tau$.

\section{Equivariant characteristic classes and Integration on elliptic cohomology}\label{mugamma}
We now apply the Chern Weil homomorphism to define the needed equivariant forms
$\mu_\gamma$ for $\tau$-integration.

\begin{definition}
Let $M$ be a compact oriented manifold with an action of $G=S^1$.
Consider $\CN\to M$  an equivariant real vector bundle with typical fiber $N$ .
We denote by
 ${\rm C_q}(\CN)(z)\in H_{[G]}(M)$ the characteristic class associated
by the Chern Weil homomorphism to the function
$C_1(\tau, N)(R)={\rm Tr}(e^R, W_{1,q}(N\otimes_\R \C))$.

If $\CN=TM$, we denote the class by ${\rm C_q}(M)(z)$.

\end{definition}

If $G$ acts trivially on $M$, then
${\rm C_q}( M)(z)={\rm C_q}( M)(0)$
is independent of $z$ and is just the usual Chern character
of $W_{1,q}(TM\otimes_\R \C)$.

\bigskip

Let $\CN\to M$ be a $G$-equivariant even dimensional Euclidean vector bundle over $M$ with orientation
$o_\CN$.  Assume that $G$ acts trivially on $M$
and that $J_G$ produces a vertical invertible transformation $J$
of $\CN$.
Let $\gamma\in \C$ be such that $a\gamma$ is not in $L_\tau$
for all rotation numbers of $J$.

\begin{definition}
We denote by ${\rm cw}Z(\gamma,J;\tau,\CN,o_\CN)(z)\in H_{[G]}(M)$
the equivariant characteristic class associated by the Chern-Weil homomorphism
to the function $Z(\gamma,J;\tau,N,o_N)$ on $\so(N,J)$,
that is
$${\rm cw}Z(\gamma,J;\tau,\CN,o_\CN)(z)=
\frac{{\rm Tr}( e^{\gamma J+{\mathcal R}(z)},
W_{1,q}(\CN\otimes_\R \C))}
{{\rm Str}( e^{\gamma J+{\mathcal R}(z)},S_\CN,o_\CN)}.$$
\end{definition}

The equivariant form ${\rm cw}Z(\gamma,J;\tau,\CN,o_\CN)(z)$
depends only of the image of $\gamma$ in $E_\tau$.

\begin{definition}
Let $\CN\to M$ be a $G$-equivariant even dimensional Euclidean vector bundle over $M$ with orientation
$o_\CN$.
Let $k$ be an integer,  $\zeta\in SO(\CN)$ a vertical action on $\CN$
such that
$\zeta^k={\rm Id}$, commuting with the $G$-action and such that ${\rm Id}-\zeta$ is invertible.
Let $\gamma\in \C$ of order $k$ in $E_\tau$.
We denote by ${\rm cw}EM(\gamma,\zeta;\tau,\CN,o_\CN)(z)$
the equivariant characteristic class associated by the Chern-Weil homomorphism
to the function $EM(\gamma,\zeta;\tau, N,o_N)$ on $\so(N,\zeta)$.
\end{definition}

Here again the equivariant form ${\rm cw}EM(\gamma,\zeta;\tau,\CN,o_\CN)(z)$
depends only of the image of $\gamma$ in $E_\tau$.

\bigskip

We are now ready to define the needed characteristic classes for $\tau$-integration.

Recall that if  $\gamma\in E_\tau$ is not a special point,
then $M_\gamma=M^G$.
If $\gamma\in E_\tau$ is a special point of order $k$,
then $M_\gamma=M(\exp(J_G/k))$.
Since $M$ is even dimensional, oriented, and with invariant spin structure, manifolds $M_\gamma$ are even dimensional and oriented.

We start by the case where $\gamma$ is not a special point.
So $M_\gamma=M^G$.
Let $\CN$ be the normal bundle  of $M^G$ in $M$.
Then $J_G$ acts vertically, provides a (locally constant)
invertible transformation $J\in \so(\CN_m)$ of the fiber $\CN_m$,
thus an orientation $o_\CN$. Furthermore we have $a\gamma$ not in $L_\tau$ for
all rotation numbers of $J$ since $\gamma$ is not a special point.
We orient $M_\gamma$ by the quotient orientation of $M$ and of $\CN$.

\begin{definition}\label{def:nuJ}
We define
$$\mu_\gamma(z)=(-2i\pi)^{-\dim M_\gamma/2}
{\hat A}(M_\gamma)(z){\rm C_q}( M_\gamma)(z)
{\rm cw}Z(\gamma,J;\tau,\CN,o_\CN)(z).$$
Here $\CN=TM/TM_\gamma$ is the normal bundle of $M_\gamma$ in $M$.

\end{definition}

In fact since $J_G$ acts trivially on $M_\gamma$,
${\hat A}(M_\gamma)(z)={\hat A}(M_\gamma)(0)$ is the usual $\hat A$ class of $M^G$
and does not depend of $z$
and ${\rm C_q}(M_\gamma)(z)={\rm C_q}( M_\gamma)(0)$ is
the usual Chern character of the bundle $W_{1,q}(TM_\gamma\otimes_\R \C)$.

\bigskip
Consider now the case where $\gamma$ is a special point of order $k$. So $M_\gamma=M(\exp(J_G/k))$.
Let $\CN\to M_\gamma$ be the normal bundle of $M_\gamma$ in $M$.
Denote  by $\zeta$   the action of $\exp(J_G/k)$ on $\CN$ (the action is vertical).
Denote  by $\hat \zeta$   the element of ${\rm Spin}(\CN)$ above $\zeta$
provided by the spin structure.
We choose the orientation $o_{\CN}$ on $\CN$ associated   to $\hat \zeta$.
We orient $M_\gamma$  by the quotient orientation of $M$ and of $\CN$.

\begin{definition}\label{def:nuzeta}
We define
$$\mu_\gamma(z)=(-2i\pi)^{\dim M_\gamma/2}
{\hat A}(M_\gamma)(z){\rm C_q}( M_\gamma)(z)
 {\rm   cw EM}(\gamma,\zeta;\tau,\CN,o_\CN)(z).$$

 Here $\CN=TM/TM_\gamma$ is the normal bundle of $M_\gamma$ in $M$.
\end{definition}

Definitions \ref{def:nuJ} and \ref{def:nuzeta}
are very similar.
However, in the non special case, the normal bundle $\CN$ is provided with a vertical action
of $S^1$, while in the special case, we have only a vertical action of $\Z/k\Z$.
The delicate definition of $\mu_\gamma(z)$ for special element
$\gamma$  is already present in Bott-Taubes, where they construct a modified Witten series of vector bundles on $M_\gamma$  using the bundles
$W_{i,q}(V)$ for $i\in \{1,2,3,4\}$.
Here we have constructed explicitly  $\mu_\gamma(z)$
as an holomorphic  germ of equivariant form on $M_\gamma$
 (giving the Berline-Vergne index formulae for the index of Bott-Taubes series at $z=0$.)
Furthermore (if we are not mistaken) by our careful study of orientations,
 compatible fixed point formulae are automatic and the classes depend only
 of the image of $\gamma$ in $E_\tau$.

\begin{theorem}(Rosu)
Let ${\bf b}=(b_\gamma)$ be an elliptic bouquet
 defined over $U\subset E_\tau$.
Then there exists a holomorphic function $I(c)$ defined on $U$ and such that
$$\int_{M_\gamma}b_\gamma(z) \mu_\gamma(z)=I(c)(\gamma+z)$$
for any $\gamma\in U$ and $z$ sufficiently small.
\end{theorem}

 \begin{proof}
The proof follows the scheme of the proof of  Theorem \ref{theo:Kint}.
If  the special points contained in $U$ are all reals (that is $\gamma \in \R$),
the theorem above follows
from the consistency of bouquet $K$-integration  of sections of
$\CSK_G(M)$.
This is the starting point of the proof of Liu  \cite{Liu} of the rigidity theorem.

So let $\gamma\in \C$.
Let $$I_\gamma(z)=\int_{M_\gamma}b_\gamma(z) \mu_\gamma(z).$$
We need to prove that, for $y$ small,
$I_\gamma(y+z)=I_{\gamma+y}(z)$
 as an identity of germs at $z=0$.

If $\gamma$ is not special, from the definition of an elliptic bouquet,
 it is sufficient to prove
that, for $y$ small, we have the identity
$\mu_\gamma(y+z)=\mu_{\gamma+y}(z)$ as an identity in $H_{[G]}(M^G)$.
But this is clear:
in fact ${\hat A}(M_\gamma)(z){\rm C_q}( M_\gamma)(y+z)$
is independent of $y$ and $z$,
and $${\rm cw}Z(\gamma,J;\tau,\CN,o_\CN)(y+z)
=\frac{{\rm Tr}( e^{\gamma J} e^{\CR(y+z)},W_{1,q}(\CN\otimes_\R\C))}
{{\rm Str}( e^{\gamma J} e^{\CR(y+z)},S_\CN,o_\CN)}$$
is equal to
$$\frac{{\rm Tr}(e^{(\gamma+y) J} e^{\CR(z)},W_{1,q}(\CN\otimes_\R\C))}
{{\rm Str}( e^{(\gamma+y) J} e^{\CR(z)},S_\CN,o_\CN)}$$
since $J_G$ acts trivially on $M^G$, using again \ref{eq:RS}.

The only delicate part of the proof is
when  $\gamma$ is special of order $k$.
Let $\zeta=\exp(J_G/k)$,
then  $M_\gamma=M(\zeta)$, and for $y\neq 0$ small,
$M_{\gamma+y}=M^G$.

Let $\CN_0$ be the normal bundle of $M^G$ in $M_\gamma$.
Let $\CN_1\to M_\gamma$  be the normal bundle of $M_\gamma$ in $M$ oriented by
$\hat \zeta$,
and $\CN\to M^G$ be the normal bundle of $M^G$ in $M$ oriented by $J$.
Let $o_{\CN_0}$ the quotient orientation $o_{\CN}/o_{\CN_1}$.
This is also the orientation of the orientation we have defined on $TM_\gamma$ and $TM^G$.

 Decompose
$$TM|_{M^G}={\mathcal T}_0\oplus \CN_{0}\oplus \CN_1|_{M^G},$$
where ${\mathcal T}_0$ it the tangent bundle $T(M^G)$.

From the localization formula and the definition of the sheaf
${\mathcal Ell}_G(M,\tau)$,
it is sufficient to verify
that
$$(-2\pi)^{rk \CN_0/2}\mu_\gamma(y+z)|_{M^G}
{\rm Eul}^{-1}(\CN_0,o_{\CN_0})(y+z)=\mu_{\gamma+y}(z)$$
as an identity in $H_{[G]}(M^G)$.
So it is sufficient to prove that
$$\left({\hat A}(M_\gamma)(y+z)  {\rm C_q}( M_\gamma)(y+z)\right)|_{M^G}
{\rm EM}(\gamma,\zeta;\tau, \CN_1,o_\CN)(y+z)
{\rm Eul}^{-1}(\CN_0,o_{\CN_0})(y+z)$$
is equal to
$$(-i)^{\dim \CN_0/2}{\hat A}(M^G)(0)
{\rm C_q}( M^G)(0)
{\rm EM}(\gamma+y,\zeta;\tau, \CN,o_\CN)(z).$$

We use  $${\hat A}(M_\gamma)(y+z)|_{M^G}=
{\hat A}(M^G)(0){\hat A}(\CN_0)(y+z)$$
$${\rm C_q}( M_\gamma)(y+z)|_{M^G}=
{\rm C_q(M^G)(0)} {\rm C_q}(\CN_0)(y+z)$$
and $${\hat A}(\CN_0)(y+z){\rm C_q}(\CN_0)(y+z){\rm Eul}^{-1}(\CN_0,o_{\CN_0})(y+z)=
{\rm cwZ}(y,J;\tau,\CN_0,o_{\CN_0})(z)$$
(deduced from Equation \ref{eq:eqjLchi} by the Chern-Weil homomorphism).

So factoring ${\hat A}(M^G)(0){\rm C_q}(M^G)(0)$,
it is sufficient to prove
that
$${\rm cwZ}(\gamma,J;\tau,\CN_0,o_{\CN_0})(y+z)
{\rm cw EM}(\gamma,\zeta;\tau, \CN_1,o_{\CN_1})(y+z)$$
 is equal to
 $${\rm cwZ}(\gamma+y,J;\tau,\CN,o_{\CN})(z).$$

The various equivariant curvatures $\CR(y+z)$ restricted to $M_G$ are
$yJ+\CR(z)$.

So  the equivariant form
 $$z\mapsto {\rm cwZ}(\gamma,J;\tau,\CN_0,o_{\CN_0})(y+z)$$
 is the image by the Chern Weil isomorphism
  of the invariant function
  $$R\mapsto Z(\gamma,J;\tau,N_0,o_{N_0})(yJ+R)$$
  on $\so(N,J)$.

 The equivariant form
 $$z\mapsto {\rm cw EM}(\gamma,\zeta;\tau, \CN_1,o_{\CN_1})(y+z)$$
 is the image by the Chern Weil homomorphism of the invariant function
 $$R\mapsto  EM(\gamma,\zeta;\tau, N_1,o_{N_1})(yJ+R)$$
 on $\so(N,J)$.

  The equivariant form
 $$z\mapsto {\rm cwZ}(\gamma+y,J;\tau,\CN,o_{\CN})(z)$$
 is the image by the Chern Weil homomorphism of the invariant function
 $$R\mapsto  Z(\gamma+y,J;\tau, N,o_{N})(R)$$
 on $\so(N,J)$.

 So the result follows from Theorem \ref{theo:spinequal}.

\end{proof}

\bigskip

Consider the global section ${\bf w}=(w_\gamma)$
with $w_\gamma=1$
for all $\gamma\in E_\tau$.
Then $I({\bf w})$ is a global function on $E_\tau$, so $I({\bf w})$ is constant.
The value of $I({\bf w})$ at $\gamma=0$
is exactly
$$\int_M {\hat A}(M)(z) {\rm Ch}(W_{1,q}(TM\otimes_\R \C))(z)$$
So we obtain that this integral is independent of $z$.
This is Witten rigidity theorem.

%
%
%


\begin{bibdiv}
\begin{biblist}


\bib{BGV}{book}{
   author={Berline, Nicole},
   author={Getzler, Ezra},
   author={Vergne, Mich\`ele},


   title={Heat kernels and Dirac operators},
   series={Grundlehren der Mathematischen Wissenschaften [Fundamental
   Principles of Mathematical Sciences]},
   volume={298},
   publisher={Springer-Verlag, Berlin},
   date={1992},
   pages={viii+369},
}

\bib{BV82}{article}{
   author={Berline, Nicole},
   author={Vergne, Mich\`ele},
   title={Classes caract\'{e}ristiques \'{e}quivariantes. Formule de localisation en
   cohomologie \'{e}quivariante},
   language={French, with English summary},
   journal={C. R. Acad. Sci. Paris S\'{e}r. I Math.},
   volume={295},
   date={1982},
   number={9},
   pages={539--541},
   %
}


\bib{ber-ver85}{article}{
   author={Berline, Nicole},
   author={Vergne, Mich\`ele},
   title={The equivariant index and Kirillov's character formula},
   journal={Amer. J. Math.},
   volume={107},
   date={1985},
   number={5},
   pages={1159--1190},
}

\bib{bott-taubes
}{article}{
   author={Bott, Raoul},
   author={Taubes, Clifford},
   title={On the rigidity theorems of Witten},
   journal={J. Amer. Math. Soc.},
   volume={2},
   date={1989},
   number={1},
   pages={137--186},
}
\bib{ChenHan}{article}{
   author={Chen, Xiaoyang},
   author={Han, Fei},
   title={A vanishing theorem for elliptic genera under a Ricci curvature bound},
   journal={ArXiv},
   number={2003.09897v1},
   date={2020}
}



\bib{duflo-vergne}{article}{
   author={Duflo, Michel},
   author={Vergne, Mich\`ele},
   title={Cohomologie \'{e}quivariante et descente},
   language={French},
   note={Sur la cohomologie \'{e}quivariante des vari\'{e}t\'{e}s diff\'{e}rentiables},
   journal={Ast\'{e}risque},
   number={215},
   date={1993},
   pages={5--108},
}
	


\bib{groj}{article}{
   author={Grojnowski, I.},
   title={Delocalised equivariant elliptic cohomology},
   conference={
      title={Elliptic cohomology},
   },
   book={
      series={London Math. Soc. Lecture Note Ser.},
      volume={342},
      publisher={Cambridge Univ. Press, Cambridge},
   },
   date={2007},
   pages={114--121},
}


\bib{Liu}{article}{
   author={Liu, Kefeng},
   title={On modular invariance and rigidity theorems},
   journal={J. Differential Geom.},
   volume={41},
   date={1995},
   number={2},
   pages={343--396},
}

\bib{LiuMa}{article}{
   author={Liu, Kefeng},
   author={Ma, Xiaonan},
   title={On family rigidity theorems. I},
   journal={Duke Math. J.},
   volume={102},
   date={2000},
   number={3},
   pages={451--474},
}

	

\bib{LiuMaZhang}{article}{
   author={Liu, Kefeng},
   author={Ma, Xiaonan},
   author={Zhang, Weiping},
   title={Rigidity and vanishing theorems in $K$-theory},
   journal={Comm. Anal. Geom.},
   volume={11},
   date={2003},
   number={1},
   pages={121--180},
}

\bib{rosu
}{article}{
   author={Rosu, Ioanid},
   title={Equivariant elliptic cohomology and rigidity},
   journal={Amer. J. Math.},
   volume={123},
   date={2001},
   number={4},
   pages={647--677},
}




\bib{vergneecm}{article}{
   author={Vergne, Mich\`ele},
   title={Geometric quantization and equivariant cohomology},
   conference={
      title={First European Congress of Mathematics, Vol. I},
      address={Paris},
      date={1992},
   },
   book={
      series={Progr. Math.},
      volume={119},
      publisher={Birkh\"{a}user, Basel},
   },
   date={1994},
   pages={249--295},
   review={\MR{1341826}},
}



\bib{witten}{article}{
   author={Witten, Edward},
   title={The index of the Dirac operator in loop space},
   conference={
      title={Elliptic curves and modular forms in algebraic topology},
      address={Princeton, NJ},
      date={1986},
   },
   book={
      series={Lecture Notes in Math.},
      volume={1326},
      publisher={Springer, Berlin},
   },
   date={1988},
   pages={161--181},
}
		
\end{biblist}
\end{bibdiv}
\end{document}